\RequirePackage{fix-cm}
\documentclass[smallextended]{svjour3}       
\smartqed  
\usepackage{amssymb}
\usepackage{url}
\usepackage[dvips]{graphicx}
\usepackage{graphics}
\usepackage[dvips]{color}
\usepackage{xcolor}
\usepackage[figuresright]{rotating}
\usepackage{multirow}
\usepackage[margin=1in]{geometry}
\usepackage{algorithm}
\usepackage{algorithmicx}
\usepackage{algpseudocode}
\usepackage{amsmath,dsfont}
\usepackage{amsmath}
\usepackage{amsfonts}
\usepackage{epic,eepic}
\usepackage{epsfig}
\usepackage{booktabs} 
\usepackage[toc,page,title,titletoc,header]{appendix}
\usepackage[font={small}]{caption}
\usepackage{longtable}
\usepackage{pdflscape}
\usepackage{hyperref}
\newcommand\norm[1]{\left\lVert#1\right\rVert}

\def\mS{{\mathcal S}}
\def\mU{{\mathcal U}}
\def\mO{{\mathcal O}}
\def\tO{{\mathcal {\tilde O}}}

\def\R{{\mathbb R}}
\def\smskip{\par\vskip 5 pt}

\def\QED{\hfill \quad{\bf Q.E.D.}\smskip}

\newtheorem{thm}{Theorem}
\newtheorem{lem}{Lemma}
\newtheorem{cor}{Corollary}
\newtheorem{prop}{Proposition}
\newtheorem{rem}{Remark}

\newtheorem{con}{Condition}

\newtheorem{asmp}{Assumption}

\DeclareMathOperator*{\argmin}{argmin}
\usepackage{amsfonts}
\begin{document}

\title{Cubic regularization methods with second-order complexity
guarantee based on a new subproblem reformulation
}

\titlerunning{CR methods with second-order complexity guarantee based on new subproblem reformulation}        

\author{Rujun Jiang         \and Zhishuo Zhou
       \and  Zirui Zhou 
}


\institute{Rujun Jiang \at
              School of Data Science, Fudan University, Shanghai, China \\
              \email{rjjiang@fudan.edu.cn}           
           \and
           Zhishuo Zhou \at
             School of Data Science, Fudan University, Shanghai, China\\ \email{zhouzs18@fudan.edu.cn}
             \and
             Zirui Zhou\at
             Huawei Technologies Canada, Burnaby, BC, Canada\\
             \email{zirui.zhou@huawei.com}
}

\date{Received: date / Accepted: date}

\maketitle

\begin{abstract}
The cubic regularization (CR) algorithm has attracted a lot of attentions in the literature in recent years. We propose a new reformulation of the cubic regularization subproblem.
  The reformulation is an unconstrained convex problem that requires computing the minimum eigenvalue of the Hessian.
  Then based on this reformulation, we derive a variant of the (non-adaptive) CR provided a known Lipschitz constant for the Hessian  and a variant of  adaptive regularization with cubics (ARC).
  We show that the iteration complexity of our variants  matches the best known bounds for unconstrained minimization algorithms using first- and second-order information.
  Moreover, we show that the operation complexity of both of our variants also matches the state-of-the-art bounds in the literature.
  Numerical experiments on test problems from CUTEst collection show that the ARC based on our new subproblem reformulation  is comparable to existing algorithms.
\end{abstract}

\section{Introduction}
Consider the generic unconstrained optimization problem
\begin{equation} \label{opt:nonconvex}
\min_{x \in \mathbb{R}^n} f(x) ,
\end{equation}
where $f:\R^n\rightarrow\R$ is a twice Lipschitz continuously differentiable and possibly nonconvex function.
Recently, the cubic regularization (CR) algorithm \cite{nesterov2006cubic,cartis2011adaptive} or its variants has attracted a lot of attentions  for solving problem~\eqref{opt:nonconvex}, due to its practical efficiency and elegant theoretical convergence guarantees.
Each iteration of the CR solves the following subproblem
\begin{equation}\label{opt:CRS}\tag{CRS}
\min_{s \in \mathbb{R}^n}  m(s):=\frac{1}{2} s^{\top}Hs+g^{\top}s+\frac{\sigma}{3}\norm{s}^3,
\end{equation}
where $H$ and $g$ represent the Hessian and gradient of the function $f$ at the current iterate, respectively, $\norm{\cdot}$ denotes the Euclidean $l_2$ norm, $H$ is an $n\times n$ symmetric matrix (possibly non-positive semidefinite) and $\sigma$ is a regularization parameter that may be adaptive during the iterations.
This model can be seen as a second-order Taylor expansion plus a cubic regularizer that makes the next iterate  not too far away from the current iterate.
It is well known that under mild conditions (\cite{nesterov2006cubic,cartis2011adaptive}), the CR  converges to a  point satisfying the second-order necessary condition (SONC), i.e.,
\[
\nabla f(x)=0,\quad\nabla^2 f(x)\succeq0,
\]
where $(\cdot)\succeq0$ means $(\cdot)$ is a positive semidefinite matrix. In the literature, it is of great interests to find a weaker condition than SONC, i.e.,
\begin{equation}
\label{eq:ASONC}
\|\nabla f(x)\|<\epsilon_g,\quad \lambda_{\min}(\nabla^2 f(x))\ge -\epsilon_H,\quad \epsilon_g,\epsilon_H>0,
\end{equation}
where $\lambda_{\min}(H)$ denotes the minimum eigenvalue for a matrix $H$. Condition \eqref{eq:ASONC} is often said to be $(\epsilon_g,\epsilon_H)$ stationary.

The CR algorithm was first considered by Griewank in an unpublished technical report (\cite{griewank1981modification}).
Nesterov and Polyak \cite{nesterov2006cubic} proposed the CR in a different perspective and demonstrated that it takes $\mO(\epsilon_g^{-3/2})$ iterations to find an $(\epsilon_g,\epsilon_g^{1/2})$ stationary point if each subproblem is solved exactly.
As in general the Lipschitz constant of the Hessian is difficult to estimate,  Cartis et al.~\cite{cartis2011adaptive,cartis2011adaptive2} proposed an adaptive version of the CR algorithm, called the ARC (adaptive
regularization with cubics), and showed that it admits an iteration complexity bound $\mO\left(\max\{\epsilon_{g}^{-3/2},\epsilon_H^{-3}\}\right)$ to find an $(\epsilon_g,\epsilon_H)$ stationary point, when the subproblems are solved inexactly and the regularization parameter $\sigma>0$ is  chosen adaptively.

Besides iteration complexity  $\mO(\epsilon_g^{-3/2})$, many subsequent studies proposed variants of the CR or other second-order methods that also have an \emph{operation complexity} $\tO({\epsilon_g^{-7/4}})$ (where $\tO( \cdot)$ hides the logarithm factors), with high probability, for finding an $(\epsilon_g,\epsilon_g^{1/2})$ stationary point of problem \eqref{opt:nonconvex}. 
Here, a unit operation can be a function evaluation, gradient evaluation, Hessian evaluation or a matrix vector product (\cite{curtis2021trust}). Based on the CR algorithm, Agarwal et al.~\cite{agarwal2017finding} derived an algorithm with such an operation complexity bound, where the heart of the algorithm is a subproblem solver  that returns, with high probability, an approximate solution to the problem \eqref{opt:CRS}   in $\tO(\epsilon_g^{-1/4})$ operations. 
After that, Carmon et al.~\cite{carmon2018accelerated} proposed an accelerated gradient method that also converges  to an  $(\epsilon_g,\epsilon_g^{1/2})$ stationary point with an {operation complexity} $\tO({\epsilon_g^{-7/4}})$.
Royer and Wright \cite{royer2018complexity} proposed a hybrid algorithm that combines Newton-like steps, the CG method for
inexactly solving linear systems, and the Lanczos procedure for approximately computing negative
curvature directions, which was shown to have an operation complexity $\tO({\epsilon_g^{-7/4}})$ to achieve an  $(\epsilon_g,\epsilon_g^{1/2})$ stationary point. Royer et al. \cite{royer2020newton} proposed a variant of Newton-CG algorithm with the same complexity guarantee.
Very recently, Curtis et al. \cite{curtis2021trust} considered a variant of trust-region Newton methods based on inexactly solving the trust region subproblem by the well known ``trust-region Newton-conjugate gradient" method, whose complexity also matches the-state-of-the-art.
All the above mentioned methods \cite{carmon2018accelerated,royer2018complexity,royer2020newton,curtis2021trust} converge with high probability like \cite{agarwal2017finding},  which is due to the use of randomized iterative methods for approximately computing the minimum eigenvalue, e.g., the Lanczos procedure.

Despite  theoretical guarantees,  the practical efficiency of solving \eqref{opt:CRS} heavily effects the convergence of the CR algorithm. Although it is one of  the most successful algorithms for solving \eqref{opt:CRS} in practice, the Krylov subspace method (\cite{cartis2011adaptive}) may fail to converge to the true solution of  \eqref{opt:CRS} in the hard case\footnote{For the problem~\eqref{opt:CRS}, it is said to be in the easy if the optimal solution $x^*$ satisfies  $\rho\|x^*\|>-\lambda_{\min}(A)$, and hard case otherwise.}  or close to being in the hard case.
Carmon and Duchi~\cite{carmon2018analysis} provided the first  convergence rate analysis of the Krylov subspace method in the easy case, based on which the authors further propose a CR algorithm with an {operation complexity} $\tO({\epsilon_g^{-7/4}})$  in \cite{carmon2020first}.
Carmon and Duchi   \cite{carmon2019gradient}  also showed  the gradient descent method that works in both the easy and hard cases is able to converge to the global minimizer if the step size is  sufficiently small, though the convergence rate is worse than the Krylov subspace method.
Based on a novel convex reformulation of \eqref{opt:CRS}, Jiang et al.~\cite{jiang2021accelerated} proposed an accelerated first-order algorithm   that works efficiently in practice in both the easy and  hard cases, and meanwhile enjoys theoretical guarantees of the same order with the  Krylov subspace method.

However, the methods in the literature  (\cite{royer2018complexity,carmon2018accelerated,royer2020newton,nesterov2006cubic,cartis2011adaptive,cartis2011adaptive2,agarwal2017finding,carmon2020first,jiang2021accelerated}), either somehow deviate the framework of the CR or ARC algorithms, and/or do not present good practical performance and  an $\tO({\epsilon_g^{-7/4}})$ operation complexity  simultaneously.
Our goal in this paper is to propose variants of the CR and ARC based on new subproblem reformulations that achieve the state-of-the-art complexity bounds and also remain close to the practically efficient CR and ARC algorithms.
Motivated by the reformulation in \cite{jiang2021accelerated}, we deduce a new unconstrained convex reformulation for \eqref{opt:CRS}. Our reformulation explores hidden convexity of \eqref{opt:CRS}, where similar ideas also appear in the (generalized) trust region subproblem (\cite{flippo1996duality,ho2017second,wang2017linear,jiang2019novel}). The main cost of the reformulation is computing the minimum eigenvalue of the Hessian.
We propose a variant of the CR algorithm with strong complexity guarantee. We consider the more realistic case where eigenvalues of Hessians are computed inexactly.
In this setting, we suppose the Lipschitz constant of the Hessian is given as $L$, the parameter $\sigma=L/2$ is non-adaptive, and each subproblem is also solved approximately.
We prove that our algorithm converges to an  $(\epsilon_g,\sqrt{L\epsilon_g})$ stationary point with an iteration complexity
$\mO({\epsilon_g^{-3/2}})$.
Moreover, we further show that each iteration costs
$\tO({\epsilon_g^{-1/4}})$
when the minimum eigenvalue of the Hessian is inexactly computed by the Lanczos procedure, and the subproblem, which is regularized to be strongly convex, is approximately solved by Nesterov's accelerated gradient method (NAG) \cite{nesterov2018lectures}  in each iteration. Combining the above facts, we further demonstrate
that our algorithm has an \emph{operation complexity} $\tO({\epsilon_g^{-7/4}})$ for finding an  $(\epsilon_g,\sqrt{L\epsilon_g})$ stationary point.  Based on the reformulation, we also propose a variant of the ARC with similar iteration and operation complexity guarantees, where $\sigma_k$ is adaptive in each iteration.

The remaining of this paper is organized as follows.
In Section \ref{sec:2}, we derive our unconstrained convex reformulation for \eqref{opt:CRS},   describe the CR and ARC algorithms and the basic setting, and give  unified convergence analysis for sufficient decrease of the model function in one iteration.
In Sections \ref{sec:3} and \ref{sec:4},  we give convergence analysis for the CR and ARC algorithms for finding an approximate second-order stationary point with both iteration complexity and operation complexity bounds that match the best known ones, respectively. In Section \ref{sec:5}, we compare numerical performance of  an  ARC embedded by  our reformulation   with ARCs based on existing subproblem solvers.
We conclude our paper in Section \ref{sec:6}.

\section{Preliminaries}
\label{sec:2}
The structure of this section is as follows. In Section \ref{sec:2.1}, we first propose our reformulation for the subproblem \eqref{opt:CRS}. Then in Section \ref{sec:2.2}, we mainly describe the framework of our variants of the CR and ARC algorithms and also state our convergence results.
Finally in Section \ref{sec:2.3}, we give  unified convergence analysis of one iteration progress for both the CR and ARC algorithms.

\subsection{A new convex reformulation for \eqref{opt:CRS}}
\label{sec:2.1}
In this subsection, we introduce a new reformulation for \eqref{opt:CRS} when $ \lambda_{\min}(H)<0$, i.e., the minimum eigenvalue of $H$ is negative.
First recall the reformulation proposed in \cite{jiang2021accelerated}
\begin{equation}
\label{eq:reformCOAP}
\begin{aligned}
 \min_{s,y}&~  g^{\top}s + \frac{1}{2}s^{\top}\left(H - \alpha I\right)s + \frac{\sigma}{3}y^{3/2} + \frac{\alpha}{2}y \\
 \ \, \mathrm{s.t.} &~ y \geq \|s\|^2, \quad y \geq \frac{\alpha^{2}}{\sigma^{2}},
\end{aligned}
\end{equation}
where  $\alpha = \lambda_{\min}(H)$.
However, this reformulation may be ill-conditioned and cause numerical instability when $y$ is small since the Hessian of the objective function for $y$ is $\frac{\sigma}{4}y^{-1/2}$, which approaches infinity when $y\rightarrow0$.
Unfortunately, this is the case for the CR or ARC algorithms when the iteration number $k$ becomes large.
We also found that due to this issue and that $y$ is of the same order with  $\|s\|^2$, CR or ARC based on solving subproblem \eqref{eq:reformCOAP} cannot achieve the state-of-the-art operation complexity $\tO({\epsilon_g^{-7/4}})$ for finding an  $(\epsilon_g,\sqrt{L\epsilon_g})$ stationary point.  To amend this issue, we proposed the following reformulation,
\begin{equation}
\label{eq:reform}
\begin{aligned}
 \min_{s,y}&~ \hat{m}(s,y):= g^{\top}s + \frac{1}{2}s^{\top}\left(H - \alpha I\right)s + \frac{\sigma}{3}y^3 + \frac{\alpha}{2}y^2 \\
 \ \, \mathrm{s.t.} &~ y \geq \|s\|, \quad y \geq -\frac{\alpha}{\sigma},
\end{aligned}
\tag{CRS$_r$}
\end{equation}
so that   $y$ is of the same order with $\|s\|$.

One key observation of this paper is that \eqref{eq:reform} can be simplified into a convex problem with single variable $s$, by applying partial minimization on $y$. Note that given any $s\in\R^n$, the $y$-problem of \eqref{eq:reform} is
$$ \min_{y\in\R} \left\{ \frac{\sigma}{3}y^3 + \frac{\alpha}{2}y^2: \ y\geq \|s\|, \ y\geq -\frac{\alpha}{\sigma} \right\}, $$
whose optimal solution is uniquely given by
\begin{equation}
\label{eq:ystar}
 y= \max\left\{ \|s\|, -\frac{\alpha}{\sigma}\right\}.
\end{equation}
This is because
the derivative of the objective function is $\sigma y^2+\alpha y$, satisfying
$$\sigma y^2+\alpha y=\sigma y\left(y+\frac{\alpha}{\sigma}\right)\ge0,$$
due to  the constraints $y\ge0$ and $y\ge -\frac{\alpha}{\sigma} $.
Substituting \eqref{eq:ystar} into \eqref{eq:reform}, we obtain that \eqref{eq:reform} is equivalent to
\begin{equation}
\label{eq:reform-u}
\min_{s\in\R^n} \left\{ \tilde{m}(s):= g^{\top}s + \frac{1}{2}s^{\top}(H-\alpha I)s + J_{\alpha,\sigma}(s)\right\}, \tag{CRS$_u$}
\end{equation}
where
\begin{equation}
\label{eq:def-omega}
J_{\alpha,\sigma}(s) := \frac{\sigma}{3}\left[\max\left\{ \|s\|, -\frac{\alpha}{\sigma}\right\}\right]^3 + \frac{\alpha}{2}\left[\max\left\{ \|s\|, -\frac{\alpha}{\sigma}\right\}\right]^2.
\end{equation}
In the following, we show that $J_{\alpha,\sigma}(s)$ is a convex and continuously differentiable function.
\begin{prop}\label{prop:omega}
For any $\sigma>0$ and $\alpha\in\R$, $J_{\alpha,\sigma}(s)$ is convex and continuously differentiable on $\R^n$. Moreover, we have
\begin{equation*}
\label{eq:diff-omega}
\nabla J_{\alpha,\sigma}(s) = \left[\sigma\|s\| + \alpha\right]_+\cdot s, \qquad \forall \, s\in\R^n,
\end{equation*}
where $[a]_+ = \max\{a,0\}$ for any $a\in\R$.
\end{prop}
\begin{proof}
We consider the two cases (a) $\alpha\geq 0$ and (b) $\alpha<0$ separately.
\begin{itemize}
\item[(a)] If $\alpha\geq 0$, then by $\sigma>0$, we have $\|s\| \geq -\alpha/\sigma$ for all $s\in\R^n$. Thus, $J_{\alpha,\sigma}(s)$ reduces to
$$ J_{\alpha,\sigma}(s) = \frac{\sigma}{3}\|s\|^3 + \frac{\alpha}{2}\|s\|^2, \quad \forall \, s\in\R^n. $$
It is clear that in this case $J_{\alpha,\sigma}(s)$ is convex and continuously differentiable, and
$$ \nabla J_{\alpha,\sigma}(s) = \sigma\|s\|s + \alpha s = (\sigma\|s\| + \alpha)\cdot s =  \left[\sigma\|s\| + \alpha\right]_+\cdot s, $$
where the last equality is due to $\alpha\geq 0$.
\item[(b)] Now we consider the case $\alpha<0$. Note that the following identity holds for any $\sigma>0$, $\alpha,y\in\R$:
$$ \frac{\sigma}{3}y^3 + \frac{\alpha}{2}y^2 = \frac{\sigma}{3}\left(y+ \frac{\alpha}{\sigma}\right)^3 - \frac{\alpha}{2}\left(y + \frac{\alpha}{\sigma}\right)^2 + \frac{\alpha^3}{6\sigma^2}. $$
By this, we can rewrite $J_{\alpha,\sigma}(s)$ in \eqref{eq:def-omega} as
\begin{equation}
\label{eq:new-J}
J_{\alpha,\sigma}(s) = \frac{\sigma}{3}\left[ \|s\| + \frac{\alpha}{\sigma}\right]_+^3 - \frac{\alpha}{2}\left[ \|s\| + \frac{\alpha}{\sigma}\right]_+^2 + \frac{\alpha^3}{6\sigma^2}.
\end{equation}
Note that $\|s\| + \alpha/\sigma$ is a convex function of $s$. In addition, $[\cdot]_+^3$ and $[\cdot]_+^2$ are both non-decreasing convex functions. Thus, we obtain that
$$ h_1(s):=\left[ \|s\| + \frac{\alpha}{\sigma}\right]_+^3 \quad \mbox{and} \quad h_2(s):=\left[ \|s\| + \frac{\alpha}{\sigma}\right]_+^2 $$
are convex functions. This, together with $\alpha<0$ and \eqref{eq:new-J}, implies that $J_{\alpha,\sigma}(s)$ is convex. Also, it is easy to verify that
\begin{align*}
\nabla h_1(s) & = \left\{
\begin{aligned}
& 0, \qquad \qquad \qquad \qquad \ \mbox{if} \ \|s\|\leq -\frac{\alpha}{\sigma}, \\
& 3\left(\|s\| + \frac{\alpha}{\sigma}\right)^2\cdot\frac{s}{\|s\|}, ~~ \mbox{if} \ \|s\| > - \frac{\alpha}{\sigma};
\end{aligned}\right. \\
\nabla h_2(s) & = \left\{
\begin{aligned}
& 0, \qquad \qquad \qquad \qquad \ \mbox{if} \ \|s\|\leq -\frac{\alpha}{\sigma}, \\
& 2\left(\|s\| + \frac{\alpha}{\sigma}\right)\cdot\frac{s}{\|s\|}, \quad \mbox{if} \ \|s\| > - \frac{\alpha}{\sigma}.
\end{aligned}\right.
\end{align*}
This, together with \eqref{eq:new-J}, implies that
$$ \nabla J_{\alpha,\sigma}(s) = \left\{
\begin{aligned}
& 0, \qquad \qquad \qquad  \mbox{if} \ \|s\|\leq -\frac{\alpha}{\sigma}, \\
& (\sigma\|s\| + \alpha)\cdot s, \ \ \mbox{if} \ \|s\|> -\frac{\alpha}{\sigma}. \\
\end{aligned}\right. $$
It then follows that $J_{\alpha,\sigma}(s)$ is continuously differentiable and
$$ \nabla J_{\alpha,\sigma}(s) = [\sigma\|s\| + \alpha]_+\cdot s. $$
\end{itemize}
Combining the results in cases (a) and (b), we complete the proof.
\QED\end{proof}
We immediately have the following results.
\begin{cor}
\label{cor:2.2}
The $\tilde{m}(s)$ in \eqref{eq:reform-u} is convex and continuously differentiable, and
\begin{equation*}
\label{eq:diff-tm}
\nabla \tilde{m}(s) = g + (H - \alpha I)s + \left[\sigma\|s\| + \alpha\right]_+ s.
\end{equation*}
More over, if $\sigma\|s\| + \alpha\ge0$, we have
\[
{m}(s)= \tilde{m}(s)\text{~ and ~}\nabla {m}(s)=\nabla \tilde{m}(s).
\]
\end{cor}

\subsection{Variants of the CR and the ARC algorithms and main complexity results}
\label{sec:2.2}
In this subsection, we first summarise  our variants of the CR and ARC algorithms in Algorithms \ref{alg:cr} and \ref{alg:ourarc}. Note that the only difference between Algorithms \ref{alg:cr} and \ref{alg:ourarc} is that Algorithm \ref{alg:ourarc}
has an adaptive regularizer $\sigma_k$ in the model function,  where  the Hessian Lipschitz constant $L$ is replaced by the adaptive parameter $2\sigma_k$, and thus Algorithm \ref{alg:ourarc} needs carefully choosing parameters related to $\sigma_k$.

\begin{algorithm}[!http]
\begin{algorithmic}[1]
\caption{A variant of the CR algorithm using reformulation \eqref{eq:reform-u}}
\label{alg:cr}
\Require $x_0,~\epsilon_g>0,~L>0$,  $\epsilon_E = \sqrt{L\epsilon_g}/3$ and  $\epsilon_S=\epsilon_g/9$
\For{$k=0,1,\ldots,$}
\State evaluate $g_k = \nabla f(x_k)$, $H_k = \nabla^2 f(x_k)$
\State compute an approximate eigenpair $(\alpha_k,v_k)$ such that
$\alpha_k =v_k^{\top}H_kv_k\leq \lambda_{\min}(H_k) + \epsilon_E$
\If{$\|g_k\|\leq \epsilon_g$ and $\alpha_k\geq -2\epsilon_E$}
\State return $x_k$\EndIf
\If{$\alpha_k\geq -\epsilon_E$}
\State solve the regularized subproblem approximately
\begin{equation}
\label{eq:mkr}
s_k \approx \argmin_{s\in\R^n} \left\{ m_{k}^r(s):=g_k^{\top}s + \frac{1}{2}s^{\top}(H_k+3\epsilon_E I)s + \frac{L}{6}\|s\|^3\right\},
\end{equation}
\State  $d_k=s_k$, $x_{k+1} = x_k + d_k$
\Else
\State solve the regularized subproblem approximately
\begin{equation}
\label{eq:mkrt}
s_k \approx \argmin_{s\in\R^n} \left\{ \tilde m_{k}^r(s):= g_k^{\top}s + \frac{1}{2}s^{\top}(H_k-\alpha_k I +2\epsilon_EI)s + \tilde J_k(s)\right\},
\end{equation}
\quad \quad \quad where $\tilde J_k(s)=J_{\alpha_k,L/2}(s)$
\If {$L\|s_k\|+2\alpha_k \geq 0$}
\State  $d_k = s_k$
\Else
\State  $w_k=\beta v_k$ such that $\|w_k\|=|\alpha_k|$ and $w_k^{\top}g_k\le 0$
\State  $d_k=\frac{1}{L}w_k$
\EndIf
\State  $x_{k+1} =x_k + d_k$
\EndIf
\EndFor
\end{algorithmic}
\end{algorithm}

\begin{algorithm}[!http]
\begin{algorithmic}[1]
\caption{A variant of the ACR algorithm using reformulation \eqref{eq:reform-u}}
\label{alg:ourarc}
\Require $x_0,~2>\gamma>1,~1>\eta>0,$ $\sigma_0>0$, $\epsilon_g>0$,  $\epsilon_E = \sqrt{L\epsilon_g}/3$ and  $\epsilon_S=\epsilon_g/9$
\For{$k=0,1,\ldots,$}
\State evaluate $g_k = \nabla f(x_k)$, $H_k = \nabla^2 f(x_k)$
\State compute an approximate eigenpair $(\alpha_k,v_k)$ such that $\alpha_k =v_k^{\top}H_kv_k\leq \lambda_{\min}(H_k) + \epsilon_E$
\If{$\|g_k\|\leq \epsilon_g$ and $\alpha_k\geq -2\epsilon_E$}
\State return $x_k$\EndIf
\If{$\alpha_k\geq -\epsilon_E$}
\State solve the regularized subproblem approximately
\begin{equation}
\label{eq:amkr}
s_k \approx \argmin_{s\in\R^n} \left\{ m_{k}^r(s):=g_k^{\top}s + \frac{1}{2}s^{\top}(H_k+3\epsilon_E I)s + \frac{\sigma_{k}}{3}\|s\|^3\right\},
\end{equation}
\State  $d_k=s_k$
\Else
\State solve the regularized subproblem approximately
\begin{equation}
\label{eq:amkrt}
s_k \approx \argmin_{s\in\R^n} \left\{ \tilde m_{k}^r(s):= g_k^{\top}s + \frac{1}{2}s^{\top}(H_k-\alpha_k I +2\epsilon_EI)s +  \tilde J_k(s)\right\},
\end{equation}
\quad \quad \quad where $\tilde J_k(s)=J_{\alpha_k,\sigma_{k}}(s)$
\If {$\sigma_{k}\|s_k\|+\alpha_k \geq 0$}
\State  $d_k = s_k$
\Else
\State  $w_k=\beta v_k$ such that $\|w_k\|=|\alpha_k|$ and $w_k^{\top}g_k\le 0$
\State  $d_k=\frac{1}{2\sigma_{k}}w_k$
\EndIf
\EndIf
\State   $\rho_k=\frac{f(x_k)-f(x_k+d_k)}{-m_k(d_k)}$
\If {$\rho_k\ge \eta$ or $\alpha_k<-\epsilon_E$}
\State $ x_{k+1}=x_k+d_k$, $\sigma_{k+1}=\sigma_k/\gamma $ \Comment successful iteration
\Else
\State $x_{k+1}= x_k$, $\sigma_{k+1}=\gamma\sigma_k $ \Comment unsuccessful iteration
\EndIf
\EndFor
\end{algorithmic}
\end{algorithm}

Before presenting the convergence analysis, we give some general assumptions and conditions that are widely used in the literature. We first introduce the following assumption for the objective function, which was used in \cite{xu2020newton}.
\begin{asmp}\label{asmp:1}
The function $f$ is twice differentiable with $f^*=\min_x f(x)$, and has bounded and Lipschitz continuous Hessian on the piece-wise linear path generated by the iterates, i.e., there exists $L>0$ such that
\begin{eqnarray}
\|\nabla^2f(x)-\nabla^2f(x_k) \|\le L\| x-x_k\|,\quad\forall x\in[x_k,x_k+d_k],\label{eq:hlip}
\end{eqnarray}
where $x_k$ is the $k$th iterate and $d_k$ is the $k$th update. Here $\|A\|$ denotes the operator 2-norm for a matrix $A$.
\end{asmp}
An immediate result of Assumption \ref{asmp:1} is the following well known cubic upper bound for any $s\in \R^n$ (cf. equation (1.1) in \cite{cartis2011adaptive})
\begin{equation}\label{eq:cubicub}
f(x_k+s)-f(x_k)\le g_k^{\top}s+\frac{1}{2}s^{\top}H_ks+\frac{L}{6}\|s\|^3.
\end{equation}

As in practice, it is expensive to compute the exact smallest eigenvalue, we consider the case that the smallest eigenvalue is approximately computed. Note that in line 3 of Algorithm \ref{alg:cr} (and line 4 of Algorithm \ref{alg:ourarc}), we call an approximate eigenvalue solver to find an approximate eigenvalue $\alpha_k$ and a unit vector $v_k$ such that
$$\lambda_{\min}(H_k) \le \alpha_k =v_k^{\top}H_kv_k\leq \lambda_{\min}(H_k) + \epsilon_E.$$To make the model function $\epsilon_E$-strongly convex,
we add $\frac{3}{2}\epsilon_E\|s_k\|^2$ to $m_k$ or $\epsilon_E\|s_k\|^2$ to $\tilde m_k$ (denoted by  $m_k^r$ or $\tilde m_k^r$), i.e.,
\[
m_{k}^r(s):=g_k^{\top}s + \frac{1}{2}s^{\top}(H_k+3\epsilon_E I)s + \frac{\sigma}{3}\|s\|^3
\]
and
\[
\tilde m_{k}^r(s):= g_k^{\top}s + \frac{1}{2}s^{\top}(H_k-\alpha_k I +2\epsilon_EI)s + \tilde J_k(s),
\]
where $\tilde J_k(s)=J_{\alpha_k,\sigma}(s)= \frac{\sigma}{3}\left[\max\left\{ \|s\|, -\frac{\alpha_{k}}{\sigma}\right\}\right]^3 + \frac{\alpha}{2}\left[\max\left\{ \|s\|, -\frac{\alpha_{k}}{\sigma}\right\}\right]^2$. Here we have $\sigma=\frac{L}{2}$ for Algorithm \ref{alg:cr} and $\sigma=\sigma_k$ for Algorithm \ref{alg:ourarc}. Since our reformulation is designed for the case that the smallest eigenvalue of the Hessian is negative, we solve $m_{k}^r(s)$ when the approximate smallest eigenvalue is larger than or equal to  criteria $-\epsilon_E$ and solve $\tilde m_{k}^r(s)$ otherwise.

To make algorithms more practical, we allow that the subproblems are approximately solved under certain criteria, i.e., the gradient norm of the model function is less than or equal to $\epsilon_S$.
\begin{con}
\label{con:mrgrad}
The subproblems \eqref{eq:mkr} and \eqref{eq:amkr} are approximately solved such that
\begin{equation}
\label{eq:gmkr}
\|\nabla m_k^r(s_k)\| \le \epsilon_S.
\end{equation}
The subproblems \eqref{eq:mkrt} and \eqref{eq:amkrt} are approximately solved such that
\begin{equation}
\label{eq:gmkrt}
\|\nabla \tilde m_k^r(s_k)\| \le \epsilon_S.
\end{equation}
\end{con}

\begin{rem}
We may also replace Condition \ref{con:mrgrad} by the following stopping criteria
\[
\|\nabla m_k^r(s_k)\|\le \max\left\{ \zeta\|s_k\|^2,\epsilon_S\right\}\text{ and }\|\nabla \tilde m_k^r(s_k)\|\le\max\left\{ \zeta\|s_k\|^2,\epsilon_S\right\},
\]
for some prescribed $\zeta\in(0,1)$ where similar ideas are  widely used in the literature \cite{cartis2011adaptive,cartis2011adaptive2,xu2020newton}. Such  stopping criteria has an advantage in practice if $\|s_k\|$ is large. By slightly modifying our proof, we still have an iteration complexity
$ \mO(\epsilon_g^{-3/2})$ and an operation complexity $\tO(\epsilon_g^{-7/4 })$.
\end{rem}

For simplicity of analysis, we consider the following condition for both Algorithms \ref{alg:cr} and \ref{alg:ourarc}.
We remark that the constants in the following condition may be changed slightly and we will still have the same order of complexity bounds.
\begin{con}\label{con:eps}
Set $\epsilon_E=\frac{1}{3}\sqrt{L\epsilon_g}$ and $\epsilon_S=\frac{\epsilon_g}{9}=\frac{\epsilon_E^2}{L}$.
\end{con}

From now on, we suppose that Assumption \ref{asmp:1} and Conditions \ref{con:mrgrad} and \ref{con:eps}   hold in the following of this paper.
Our first main result is that both Algorithms \ref{alg:cr} and \ref{alg:ourarc} find an $(\epsilon_g,\sqrt{L\epsilon_g})$  stationary point in at most $ \mathcal{O}\left(\epsilon_g^{-3/2}\right) $  iterations (see Theorems \ref{thm:crcplx} and \ref{thm:arc}).
Then we will show that under some mild assumptions (Assumptions \ref{asmp:ybd} and \ref{asmp:hbd}), if the eigenvalue is approximated by the Lanczos procedure and the subproblem is approximately solved by NAG,  then each iteration costs at most $\tO(\epsilon_g^{-1/4})$ operations.
Thus the operation complexity of Algorithm \ref{alg:cr} is  $\tO(\epsilon_g^{-7/4})$ (see Theorem \ref{thm:cropcomp}). Similar results also hold for the ARC and are omitted for simplicity.

\begin{rem}
Our goal is to present variants of the CR and ARC that  are close to their practically efficient versions (\cite{nesterov2006cubic,cartis2011adaptive,cartis2011adaptive2}).
Most of the existing works on the CR or ARC do not present an operation complexity $\tO(\epsilon_g^{-7/4})$ (\cite{nesterov2006cubic,cartis2011adaptive,cartis2011adaptive2,xu2020newton}), while other  existing works in the framework of the CR or ARC that prove to admit an operation complexity $\tO(\epsilon_g^{-7/4})$ (\cite{agarwal2017finding,carmon2020first}) deviate more largely form the practically efficient versions than ours.
The subproblem solver in  \cite{agarwal2017finding} requires sophisticated parameter tuning and seems hard to implement in practice.
The iteration number of each subproblem solver in \cite{carmon2020first} is set in advance, which may take additional cost in practice if the subproblem criteria is early met.
Moreover, both works are restricted to the case of known gradient and/or Hessian Lipschitz constant, and they are restricted to the CR case. On the other hand, our methods are more close to the practically efficient CR and ARC algorithms in \cite{nesterov2006cubic,cartis2011adaptive2}.
We only add an additional regularizer $\frac{3}{2}\epsilon_E\|s\|^2$ or $\epsilon_E\|s\|^2$ to the original model function in the CR or ARC, use an approximate solution as the next step in most cases (in fact related to the easy case of the subproblem), and use a negative curvature direction in the other case  (related to the hard case).
\end{rem}

\subsection{Progress in one iteration of the model function}
\label{sec:2.3}
In this subsection, we give unified analysis for the descent progress in one iteration of the models for both the CR and ARC algorithms, which will be the heart of our convergence analysis of iteration complexity for the CR and ARC algorithms.

 \begin{prop}
\label{prop:term}
If Algorithm \ref{alg:cr} terminates (at line 5) or Algorithm \ref{alg:ourarc} terminates (at line 5), then the output $x_k$ is  an $(\epsilon_g,\sqrt{L\epsilon_g})$ stationary point.
\end{prop}
\begin{proof}
Note that in  line 5 of either  Algorithm \ref{alg:cr} or   Algorithm \ref{alg:ourarc}, we have $\|g_k\|\le\epsilon_g$ and $\alpha_k\ge -2\epsilon_{E}$.
Combining $\alpha_k\leq\lambda_{\min}(H_k) + \epsilon_E$ and $\alpha_k\geq -2\epsilon_E$, we obtain $\lambda_{\min}(H_k) \geq -3\epsilon_E=-\sqrt{L\epsilon_g}$ due to Condition \ref{con:eps}. This, together with $\|g_k\|\leq \epsilon_g$, yields the desired result.
\QED\end{proof}

In the following two lemmas, we show sufficient decrease can be achieved in the case where either  $\alpha_k\geq -\epsilon_E$ or $\alpha_k< -\epsilon_E$. The proofs for both lemmas defer to the appendix.
\begin{lem}
\label{lem:bd1}
Suppose that $x_k+d_k$ is not an $(\epsilon_g,\sqrt{L \epsilon_g})$ stationary point.
Suppose Assumption \ref{asmp:1} and  Conditions \ref{con:mrgrad} and \ref{con:eps} hold and $\alpha_k\geq -\epsilon_E$ for some iteration $k$. Then for Algorithm \ref{alg:cr}, we have
\begin{equation*}
-m(d_{k})\ge \frac{\epsilon_E^3}{L^2}.
\end{equation*}
For Algorithm \ref{alg:ourarc}, we have
\begin{equation*}
-m(d_{k})\geq \min\left\{\frac{3}{\max\{\gamma,2\sigma_0/L\}+1},1\right\}\cdot\frac{\epsilon_E^3}{L^2}.
\end{equation*}
\end{lem}

\begin{lem}
\label{lem:bd2}
Suppose that $x_k+d_k$ is not an $(\epsilon_g,\sqrt{L \epsilon_g})$ stationary point.
Suppose Assumption \ref{asmp:1} and Conditions \ref{con:mrgrad} and \ref{con:eps} and in addition $1<\gamma<2$ for Algorithm \ref{alg:ourarc}, if  $\alpha_k<-\epsilon_E$. Then for  Algorithms \ref{alg:cr}, we have
\begin{equation*}
\label{eq:case2-bd2}
-m(d_{k})\geq  \frac{\epsilon_E^3}{3L^2}.
\end{equation*}
For Algorithm \ref{alg:ourarc}, we have
\begin{equation*}
\label{eq:case2-bd2}
-m(d_{k})\geq  \frac{\epsilon_E^3}{3(\max\{2\sigma_0,\gamma L\})^2}.
\end{equation*}
\end{lem}

\section{Convergence analysis for the CR algorithm}
\label{sec:3}
In this section, we first give iteration complexity analysis of the CR algorithm and then study its operation complexity in the case that the subproblem is solved by Nesterov's accelerated gradient method (NAG) and the approximate smallest eigenvalue of the Hessian is computed by the Lanczos procedure. The notation in this section follows that in Section \ref{sec:2}.

We  have the following theorem that gives a complexity bound that matches the best known bounds in the literature (\cite{nesterov2006cubic,cartis2011adaptive2,xu2020newton}).

\begin{thm}
\label{thm:crcplx}
Given Assumption \ref{asmp:1} and Conditions \ref{con:mrgrad} and \ref{con:eps},   Algorithm \ref{alg:cr} finds an $(\epsilon_g,\sqrt{L\epsilon_g})$  stationary point in at most $ \mathcal{O}\left(\epsilon_g^{-3/2}\right) $  iterations.
\end{thm}
\begin{proof}
First note that \eqref{eq:cubicub} implies $f(x_k+d_k)-f(x_k)\le m(d_k)$.
Combining Proposition \ref{prop:term}, Lemmas \ref{lem:bd1} and \ref{lem:bd2} and  Condition \ref{con:eps}, and noting $\sigma=L/2$ for Algorithm \ref{alg:cr}, we have
\[
f(x_k)-f(x_k+d_k)\ge \frac{1}{3L^2}\epsilon_E^3.
\]
Adding the above inequalities from $0$ to $T$, we have
\[
f(x_0)-f(x_T)\ge \frac{T}{3L^2}\epsilon_E^3.
\]
Noting that $f(x)$ is lower bounded from Assumption \ref{asmp:1}, we complete the proof.
\QED\end{proof}
Next we give an estimation for the cost of each iteration and thus obtain the total operation complexity. Particularly, we invoke a backtracking line search version of NAG \cite{nesterov2018lectures} (described in Algorithm \ref{alg:NAG}) to approximately solve the subproblems \eqref{eq:mkr} and \eqref{eq:mkrt} in Algorithm \ref{alg:cr}.
Note that   the objective functions $m_k^r$ and $\tilde m_k^r$ in  \eqref{eq:mkr} and \eqref{eq:mkrt} are both $\epsilon_E$-strongly convex.
 In Algorithm \ref{alg:NAG}, $h$ stands for either  $m_k^r$ or $\tilde m_k^r$.

\begin{algorithm}[!ht]
\begin{algorithmic}[1]
\caption{NAG for minimizing $m$ strongly convex smooth functions $h(z)$}
\label{alg:NAG}
\Require $h$, $\nabla h$, $t_{0}>0$, $\theta_{0}\in(0,1]$, $\beta\in(0,1)$, initial point $z_0 \in \R^n$ \For {$l=0,1,...$}
\If{$l\ge1$}
\State $t_l=t_{l-1}$\Comment initial step size for the $l$th iteration
\State $\gamma_l=\frac{\theta_{l-1}^2}{t_{l-1}}$
\State $\frac{\theta_l^2}{t_l}=(1-\theta_l)\gamma_l+m\theta_l$
\EndIf
\State $y=z_l+\frac{\theta_l\gamma_l}{\gamma_l+m\theta_l}(v_l-z_l)$\quad($y=z_0$ for $l=0$)
\State $z_{l+1}=y-t_l\nabla h(y)$
\While {$h(y-t_l\nabla h(y))>h(y)-\frac{t_{l}}{2}\|\nabla h(y)\|^2$}
\State $t_{l}=\beta t_{l}$
\State $z_{l+1}=y-t_l\nabla h(y)$
\EndWhile
\State $v_{l+1}=z_l+\frac{1}{\theta_l}(z_{l+1}-z_l)$
\EndFor
\end{algorithmic}
\end{algorithm}

\begin{asmp}\label{asmp:ybd}
 Suppose that for any extrapolated point $y$,  and any $t\in(0,t_0]$, there exists an upper bound for  $\|y-t\nabla h(y) \|$, i.e., there exists $M_n>0$ such that $\|y-t\nabla h(y)\|\le M_{n}$.
 Moreover, we assume \[\|z_l\|\le M_{n}, ~\forall l\ge0\text{ and }\|z^*\|\le M_{n},\]
 where $z_l$ is given in Algorithm \ref{alg:NAG} and $z^*$ is  the optimal solution  of the subproblem  \eqref{eq:mkr} or \eqref{eq:mkrt}.
\end{asmp}
The above assumption is easy to met. Indeed, $z_l$ is bounded because  $h(z_l)$ is bounded from standard analysis for NAG (e.g., equation \eqref{eq:nagbd}), $h$ is  strongly convex and ${\rm dom}(h)=\R^n$, which is the case for $m_k^r$ and $\tilde m_k^r$. Meanwhile, $y-t\nabla f(y)$ is bounded,   if, noting that   $y$ is a linear combination of $z_l$ and $z_{l-1}$, $\frac{\theta_l\gamma_l}{\gamma_l+m\theta_l}$ and $\frac{1}{\theta_l}$ are bounded constants, which is quite mild and  holds in most practical cases.

We also make the following assumption that is widely used in the literature (\cite{cartis2011adaptive,xu2020newton}).
\begin{asmp}
\label{asmp:hbd}
Suppose  the Hessian $H_k$ is bounded in  each iteration of Algorithm \ref{alg:cr},
i.e., there exists some constant $M_H>0$ such that
\[\|\nabla ^2f(x_k)\|\le M_{H}.\]
\end{asmp}

The above two assumptions, together with  Assumption \ref{asmp:1}, yield  the following Lipschitz continuity result on the gradient $\nabla h(y)$. \begin{lem}
\label{lem:subplip}
Under Assumptions \ref{asmp:1}, \ref{asmp:ybd} and \ref{asmp:hbd},  the gradients of $m_k^r$ and $\tilde m_k^r$
are $L_S:=2M_H+3\epsilon_E+L M_n$ Lipschitz continuous on the line path $[y,y- t_0\nabla h(y) ]$ for any extrapolated point $y$ in line 7 and the line path $[z_l,z^*]$ in Algorithm \ref{alg:NAG}.
\end{lem}
\begin{proof}
It suffices to show that for any $p$ and $q$ with $\|p\|\le M_n$ and $\|q\|\le M_n$, we have
\[
\|\nabla h(p)  - \nabla h(q)\|\le L_S\|p-q\|,
\]
where $h$ stands for either $m_k^r$ or $\tilde m_k^r$.
From the definition of $m_k^r$, we have
\begin{eqnarray*}
\|\nabla m_k^r(p)-\nabla m_k^r(q)\|
&=&\left \|H_k(p-q)+3\epsilon_E(p-q)+\frac{L}{2} \|p\|p-\frac{L}{2} \|q\|q\right\|\\
&\le&\|H_k(p-q)\|+3\epsilon_E\|p-q\|+ \frac{L}{2}\left\| (\|p\|p- \|p\|q)+(\|p\|q-\|q\|q) \right\| \\
&\le&\left(\|H_k\|+3\epsilon_E+ \frac{L}{2}(\|p\|+\|q\|)\right)\|p-q\| \\
&\le&(M_H+3\epsilon_E+L M_n)\|p-q\|,
\end{eqnarray*}
where the last inequality follows from Assumptions \ref{asmp:ybd} and \ref{asmp:hbd}.

To show the Lipschitz continuity of $\nabla \tilde  m_k^r$, we need to consider three cases:
\begin{enumerate}
\item Both $\|p\|+\frac{2\alpha_k}{L}\ge0$ and  $\|q\|+\frac{2\alpha_k}{L}\ge0$. In this case, both $\tilde \nabla m_k^r(p)=\nabla m_k^r(p)-\epsilon_Ep$ and $\tilde\nabla m_k^r(q)=\nabla m_k^r(q)-\epsilon_Eq$. With a similar analysis to the previous proof, it is easy to show $\tilde \nabla m_k^r(p)$ is  $(M_H+2\epsilon_E+L M_n)$ Lipschitz continuous.
\item Both $\|p\|+\frac{2\alpha_k}{L}\le 0$ and  $\|q\|+\frac{2\alpha_k}{L}\le0$. It is trivial to see $\tilde \nabla m_k^r(p)$ is  $(M_H+2\epsilon_E)$ Lipschitz continuous as $\nabla \tilde J_k(p)=\nabla \tilde J_k(q)=0$.
\item Either (i) $\|p\|+\frac{2\alpha_k}{L}>0$,  $\|q\|+\frac{2\alpha_k}{L}\le0$ or (ii) $\|p\|+\frac{2\alpha_k}{L}\le0$,  $\|q\|+\frac{2\alpha_k}{L}>0$. Due to symmetry, we only prove the first case. From Proposition \ref{prop:omega}, We have
\begin{eqnarray*}
\| \nabla \tilde m_k^r(p)-\nabla\tilde m_k^r(q)\|
&=&\left \|(H_k-\alpha_kI)(p-q)+2\epsilon_E(p-q)+\left(\frac{L}{2}\|p\|+\alpha_k\right)p-0\right\|\\
&\le&\|(H_k-\alpha_kI)(p-q)\|+2\epsilon_E\|p-q\|+ \left\|(\frac{L}{2}\|p\|+\alpha_k)p \right\| \\
&\le&\left(2\|H_k\|+3\epsilon_E\right)\|p-q\|+ \left\|\left(\frac{L}{2}\|p\|+\alpha_k\right)p-\left(\frac{L}{2}\|q\|+\alpha_k\right)p \right\| \\
&\le&\left(2M_H+3\epsilon_E+\frac{L}{2} M_n\right)\|p-q\|,
\end{eqnarray*}
where  in the second inequality we use $\|H_k-\alpha_k I\|\le 2\|H_k\|+\epsilon_E$ as  $\lambda_{\min}(H_k)+\epsilon_E\ge\alpha_k\ge \lambda_{\min}(H_k)$ and $2L\|q\|+\alpha_k\le 0$, and the last inequality follows from Assumptions \ref{asmp:ybd} and \ref{asmp:hbd}.\QED
\end{enumerate}
\end{proof}

Now let us give an estimation for the iteration complexity of Algorithm \ref{alg:NAG} to achieve a point such that $\|\nabla h(z_l)\|\le \epsilon_S$.
\begin{lem}
\label{lem:compNAG}
Suppose Algorithm \ref{alg:NAG} is used as subproblem solvers for \eqref{eq:mkr} and \eqref{eq:mkrt}. Given Conditions \ref{con:mrgrad} and \ref{con:eps} and Assumptions \ref{asmp:1}, \ref{asmp:ybd} and \ref{asmp:hbd}, Algorithm \ref{alg:NAG}  takes at most
$
\tO\left(\epsilon_E^{-1/2}\right)$ iterations to achieve a point such that $\|\nabla h(z_l)\|\le \epsilon_S=\frac{ \epsilon_E^2}{L}$. Moreover, the cost in each iteration is dominated by two matrix vector products.
\end{lem}
\begin{proof}
Note that either $m_k^r$ or $\tilde m_k^r$ is $\epsilon_E$-strongly convex due to the definitions, and $L_S$-smooth due to Lemma \ref{lem:subplip}.
From complexity results of NAG in  \cite{nesterov2018lectures,vandenberghe2021accelerated}, we obtain that \begin{equation}\label{eq:nagbd}
h(z_l)-h^*\le \Pi_{i=1}^{l-1}(1-\sqrt{\epsilon_E t_i})C,\end{equation}
where $C=\left((1-\theta_0)(h(x_0)-h^*)+\frac{\theta_0^2}{2t_0}\|x_0-x^*\|^2\right),$
and $t_i\ge \min\{t_0,\beta/L_{S}\}$. Thus \eqref{eq:nagbd} further yields
\[
h(z_l)-h^*\le \left(1-\sqrt{\epsilon_E\min\{t_0,\beta/L_{S}\}}\right)^{k-1}C.
\]
Therefore it takes at most $T=\mO\left(\sqrt{\frac{1}{\epsilon_E}}\log\frac{1}{\epsilon_h}\right)$ to achieve a solution such that $h(z_T)-h^*\le \epsilon_h$.

From the  $L_S$ smoothness  of $m_k^r$ and $\tilde m_k^r$ along the line $[z_l,z^*]$ (due to Lemma \ref{lem:subplip}),  we  further have
\[\frac{1}{2L_S}\|\nabla h(z_l)\|^2\le h(z_l)-h^*,\quad\forall k\ge0.\]
Thus by letting $\epsilon_h=\epsilon_S^{2}/2L_S$, we have
$$\|\nabla h(z_l)\|\le\sqrt{2L_S\epsilon_h}=\epsilon_S=\frac{1}{9}\epsilon_g=\frac{ \epsilon_E^2}{L}.$$
Hence the iteration complexity for $\|\nabla h(z_T)\|\le \epsilon_S$ is
$\mO\left(\sqrt{\frac{1}{\epsilon_E}}\log\frac{1}{\epsilon_E}\right)= \tO\left(\epsilon_E^{-1/2}\right)$.

Note that each iteration of Algorithm \ref{alg:NAG} requires one gradient evaluation of $\nabla h(y)$ according to the expression of $m_k^r$ and $\tilde m_k^r$, where the most expensive operator is the Hessian vector product $H_ky$. Then the function evaluation of $h(y)$ is cheap if we store $H_ky$. Meanwhile,
to compute $m_k(y-t_l\nabla h(y))$ for different $t_{l}$, we have\[
m_k(y-t_l\nabla h(y))=g_{k}^{T}y-t_lg_{k}^{T}\nabla h(y)+\frac{1}{2}y^{T}H_ky-t_ly^{\top}H_k\nabla h(y)+\frac{t_l^2}{2}\nabla h(y)^{\top}H_k\nabla h(y)+\frac{L}{6}\|y-t_l\nabla h(y)\|^3,
\]
which costs $\mO(1)$ if $H_k\nabla h(y),~g_{k}^{T}y,~g_{k}^{T}\nabla h(y),~y^{T}H_ky,~y^{\top}H_k\nabla h(y),~\|y\|,~y^{\top}\nabla h(y)$ and $\|\nabla h(y)\|$ are provided (using $\|y-t_{l}\nabla h(y)\|^2=\|y\|^2-2t_{l}y^{\top}\nabla h(y)+\|t_{l}\nabla h(y)\|^2)$.
Note that in the $l$th iteration, we have  $t_{0}\ge t_l\ge \min\{\beta/L_S,t_0\}$.
We thus at most do $\mO(1)$ searches for $\beta$. So in one iteration, the total cost is two matrix vectors products and $\mO(n)$ other operations. With a similar analysis, the same complexity result holds for $\tilde m_k^r$.
\QED\end{proof}

The following lemma shows a well known result that the smallest eigenvalue of a given matrix can be computed efficiently with high probability.
\begin{lem}[\cite{kuczynski1992estimating} and Lemma 9 in \cite{royer2018complexity}]
\label{lem:lanczos}
Let $H$ be a symmetric matrix satisfying $\|H\|\le U_{H}$ for some $U_{H}>0$, and $\lambda_{\min}(H)$ its minimum eigenvalue.
Suppose that the Lanczos procedure is applied to find the largest eigenvalue of $ U_{H}I-H$
starting at a random vector distributed uniformly over the unit sphere. Then, for any
$\epsilon>0$ and $\delta\in(0, 1)$, there is a probability at least $1-\delta$ that the procedure outputs a
unit vector $v$ such that
$v^{\top}Hv\le\lambda_{\min}(H) +\epsilon$
in at most
$\min\left\{n,\frac{\log(n/\delta^2)}{2\sqrt2}\sqrt{\frac{ U_{H}}{\epsilon}}\right\}$
iterations.
\end{lem}

Now we are ready to present the main result in this section that Algorithm \ref{alg:cr} has an operation complexity $\mathcal{\tilde O}$ $(\epsilon_g^{-7/4})$.
\begin{thm}
\label{thm:cropcomp}
Suppose the approximate eigenpair in line 3 of Algorithm \ref{alg:cr} is computed by the Lanczos Procedure, and  subproblems \eqref{eq:mkr} and \eqref{eq:mkrt} are approximately solved by Algorithm \ref{alg:NAG}. Under Conditions \ref{con:mrgrad} and \ref{con:eps} and Assumptions \ref{asmp:1}, \ref{asmp:ybd} and \ref{asmp:hbd},  the algorithm finds an $(\epsilon_g,\sqrt{L\epsilon_g})$ stationary point with high probability, and in this case the operation complexity of Algorithm \ref{alg:cr} is $\mathcal{\tilde O}$ $(\epsilon_g^{-7/4})$.
\end{thm}
\begin{proof}
First note that the iteration complexity is $\mO(\epsilon_g^{-3/2})$ due to Theorem \ref{thm:crcplx}.

At each iteration, if the subproblems are approximately solved in line 8 or 11 in Algorithm \ref{alg:cr}, the subproblem iteration complexity is
$\tO(\epsilon_E^{-1/2})=\tO(\epsilon_g^{-1/4})$ because that  $\epsilon_E=\sqrt{L\epsilon_g}/3$, and that the dominated cost is $\tO(\epsilon_g^{-1/4})$  matrix vector products, thanks to Lemma \ref{lem:compNAG}.

Another cost at each iteration is inexactly computing the smallest eigenvalue.
Note that the failure probability of the  Lanczos procedure is only in the ``log factor" in the complexity bound. Hence, for any given $\delta'\in(0,1)$, in the Lanczos procedure we can use a very small $\delta$ like $\delta=\delta'/T$, where $T$ is the total iteration number bounded by $\mO(\epsilon_g^{-3/2})$. Then from the union bound, the full Algorithm \ref{alg:cr} finds an $(\epsilon_g,\sqrt{L\epsilon_g})$ stationary point with probability $1-\delta'$.
From Lemma \ref{lem:lanczos}, it takes $\tO(\epsilon_E^{-1/2})=\tO(\epsilon_g^{-1/4})$ matrix vector products to achieve an $\epsilon_E$ approximate eigenpair,  with   probability at least $1-\delta'$.

As the iteration complexity of Algorithm \ref{alg:cr} is $\mO(\epsilon_g^{-3/2})$ and each iteration takes $\tO(\epsilon_g^{-1/4})$ unit operations, we conclude that the operation complexity is  $ \mathcal{\tilde O}(\epsilon_g^{-7/4})$.
%
\QED\end{proof}

\section{Convergence analysis for the ARC algorithm}
\label{sec:4}
In this section, we first  show that the ARC algorithm also has an iteration complexity  $\mO\left({\epsilon_g}^{-3/2}\right)$ for finding an $(\epsilon_g,\sqrt{L\epsilon_g})$ stationary point. Then we will briefly analyze its operation complexity in the case that the subproblem is solved by NAG and the approximate smallest eigenvalue of the Hessian is computed by the Lanczos procedure. The notation in this section follows that in Section \ref{sec:2}.

To show the iteration complexity of the ARC algorithm is still $\mO(\epsilon_g^{-3/2})$, the key proof here is that we need to counter the iteration number for  successful steps. Specifically, we need the following lemma that shows when $\sigma_k$ is large enough, the iteration must be successful.
\begin{lem}\label{lem:suc}
Suppose Assumption \ref{asmp:1} holds. If $\sigma_k\ge L/2$ and $m_{k}(d_k)<0$, then the $k$th iteration is successful.
\end{lem}
\begin{proof}
By \eqref{eq:cubicub} and $\sigma_k\ge L/2$, we have
\begin{align*}
f(x_{k}+d_{k}) - f(x_k) & \leq g_k^{\top}d_k + \frac{1}{2}d_k^{\top}H_kd_k + \frac{L}{6}\|d_k\|^3 \notag\\
 & \leq g_k^{\top}d_k + \frac{1}{2}d_k^{\top}H_kd_k + \frac{\sigma_k}{3}\|d_k\|^3 \label{eq:negsuc}\\
& =m_k(d_k)<0.\notag
\end{align*}
This yields
$\rho_k=\frac{f(x_k)-f(x_{k}+d_{k})}{-m_k(d_k)}\ge1>\eta.$ Thus the $k$th iteration is successful.
\QED\end{proof}

%

The following lemma shows that the adaptive regularizer is bounded above.
\begin{lem}\label{lem:gamma}
Suppose Assumption \ref{asmp:1} holds. Then $\sigma_k\le \max\{\sigma_0,\gamma L/2\},$ $\forall k\ge0$.
\end{lem}
\begin{proof}
Suppose the $k$th iteration is the first unsuccessful iteration such that  $\sigma_{k+1}=\gamma\sigma_k\ge\gamma L/2$, which implies $\sigma_k\ge L/2$. However, from Lemma \ref{lem:suc}, we know that the $k$th iteration must be successful and thus $\sigma_{k+1}=\sigma_k/\gamma<\sigma_k$, which is a contradiction.
\QED\end{proof}

Now we are ready to present our main convergence result of Algorithm \ref{alg:ourarc}, which is of the same order with the best known iteration bound (\cite{cartis2011adaptive2,xu2020newton}).
\begin{thm}
\label{thm:arc}
Suppose that Assumption \ref{asmp:1} and Conditions \ref{con:mrgrad}
 and \ref{con:eps} hold, and $\max\{\frac{\sigma_0}{L},\frac{\gamma }{2}\}\le 1$.
Then  Algorithm  \ref{alg:ourarc} takes
$T\le \mO(\epsilon_g^{-3/2})$ iterations to find an $(\epsilon_g,\sqrt{L\epsilon_g})$ stationary point.
\end{thm}
\begin{proof}

Note that
\begin{equation}
\label{eq:tsu}
T=|\mS|+|\mU|,
\end{equation} where $\mS$ is the index set of successful iterations and $\mU$ is the index set of unsuccessful iterations.
Here, $|\mathcal A|$ denotes the cardinality of a set $\mathcal A$.
Since  $\sigma_T=\sigma_0\gamma^{|\mU|-|\mS|}$ and $\sigma_T\le \max\{\sigma_0,\gamma L/2\}$ due to Lemma \ref{lem:gamma}, we have
\begin{equation}
\label{eq:us}
|\mU|\le \max\left\{0,\log_\gamma\left(\frac{\gamma L}{2\sigma_0}\right)\right\}+|\mS|.
\end{equation}
Note also that $\mS=\mS_1\cup\mS_2$, where
\begin{eqnarray*}
\mS_1&:=&\{k\in\mS:\|\nabla f(x_k+d_k)\|\le \epsilon_g \text{ and }\lambda_{\min}(H_{k+1})\ge -\sqrt{L\epsilon_g}\},\\
\mS_2&:=& \mS\setminus\mS_1.
\end{eqnarray*}
Now we have
\begin{eqnarray*}
f(x_0)- f^*&\ge& \sum_{k=0}^\infty f(x_k)-f(x_{k+1})=\sum_{k\in \mS} f(x_k)-f(x_{k+1})\\
&\ge&\sum_{k\in \mS_{2}} f(x_k)-f(x_{k+1})\\
&\ge&\sum_{k\in \mS_{2}} -\eta m_k(d_k)\\
&\ge&\sum_{k\in \mS_{2}}\eta \min\left\{\frac{3}{\max\{\gamma,2\sigma_0/L\}+1},1,\frac{1}{3(\max\{2\sigma_0/L,\gamma\})^2}\right\}\cdot\frac{\epsilon_E^3}{L^2}
\end{eqnarray*}
where the fifth inequality follows from Lemmas \ref{lem:bd1} and \ref{lem:bd2}.
This, together with $\epsilon_E=\sqrt{L\epsilon_g}/3$,  gives
\[|\mS_2|\le   \mO(\epsilon_g^{-3/2}). \]
It is obvious that $|\mS_1|=1$ as the algorithm terminates in one iteration.
Then we have
$$|\mS|=|\mS_1|+|\mS_2|\le \mO\left(\epsilon_g^{-3/2}\right).$$
This, together with \eqref{eq:tsu} and \eqref{eq:us}, gives $T\le \mO(\epsilon_g^{-3/2}) $.
\QED\end{proof}

In fact, with a similar analysis to Section \ref{sec:3}, we can show that the operation complexity for Algorithm \ref{alg:ourarc} is still $\tO(\epsilon_g^{-7/4})$ to find an $(\epsilon_g,\sqrt{L\epsilon_g})$ stationary point under mild conditions with high probability, if NAG and the Lanczos procedure are used in each iteration.
This is because the matrix vector product number in each iteration of Algorithm \ref{alg:ourarc} is still $\tilde\mO(\epsilon_g^{-1/4})$. Two key observations for proving the $\tilde\mO(\epsilon_g^{-1/4})$ bound of NAG are that $\sigma_{k}$ is upper bounded by constants as shown in Theorem \ref{thm:arc}, and that the subproblems are still $\epsilon_E$-strongly convex and Lipschitz smooth.
The Lipschitz smoothness follows from a similar technique with Lemma \ref{lem:subplip} under Assumptions \ref{asmp:ybd} and \ref{asmp:hbd}.


\section{Numerical experiments}
\label{sec:5}
This section mainly shows the effects of our new subproblem reformulation without the additional regularizer $\epsilon_E\|s\|^2$ for the ARC algorithm. We did numerical experiments among ARC algorithms (\cite{cartis2011adaptive})  with different subproblem solvers and compared their performance.
We point out that we do not directly implement Algorithm \ref{alg:ourarc} since it is practically inefficient if we compute the minimum eigenvalue of the Hessian at every iteration.
Particularly, in Algorithm~\ref{alg:arc}, we only call a  subproblem solver based on reformulation \eqref{eq:reform-u} if a prescribed condition is met.

Let  $f$ denote the objective function, $g_k$ denote the gradient $\nabla f(x_k)$ and $H_k$ denote the Hessian $\nabla^2f(x_k)$.
In Algorithm~\ref{alg:arc}, we use the Cauchy point $s_k^C$ (as in \cite{cartis2011adaptive}) as the initial point of the subproblem solver in each iteration:
\[s_k^C=-\alpha_k^Cg_k\text{ and }\alpha_k^C=\argmin_{\alpha\in\R_+} m_k(-\alpha g_k),\]
which is obtained by globally minimizing $m_k(s)=g_k^s+s^{\top}H_ks+\frac{\sigma_k}{3}\|s\|^3$ along the current negative gradient direction.
Let $\mathcal A$ denote an arbitrary solver for \eqref{opt:CRS},
$\mathcal A_r$ denote an arbitrary solver for the constrained reformulation \eqref{eq:reform}
and $\mathcal A_u$ denote an arbitrary solver for the unconstrained reformulation \eqref{eq:reform-u}. Because  the subproblem solver $\mathcal A_u$ (or $\mathcal A_r$) are designed for cases where $H_k$ is not positive semidefinite,
and the Cauchy point is a good initial point when the norm of the gradient is large,
we call the solver $\mathcal A_u$ (or $\mathcal A_r$) if the following condition is met:
\begin{equation}
  \label{eq:solvercon}
  \|g_k\|\le\max\left(f(x_{k}),1\right)\cdot \epsilon_1\quad\text{ and }\quad\lambda_{\min}(H_k)<-\epsilon_2,
\end{equation}
where $\epsilon_1$ and $\epsilon_2$ are some small positive real numbers and $\lambda_{\min}(H_k)$ is the minimum eigenvalue of $H_k$.
If condition \eqref{eq:solvercon} is not met, we call $\mathcal A$ to solve the model function directly.
We only accept the (approximate) solution $s_k$    if  $m_{k}(s_k)$ is smaller than that  $m_{k}(s_k^C)$; otherwise the Cauchy point $s_k^C$ is used.
This guarantees that Algorithm \ref{alg:arc} converges to a first-order stationary point under mild conditions  (\cite[Lemma 2.1]{cartis2011adaptive}).

\begin{algorithm}[!ht]
\begin{algorithmic}[1]
\caption{ARC using convex reformulation }
\label{alg:arc}
\Require $x_0,~\gamma_2\ge \gamma_1>1,~1>\eta_2\ge\eta_1>0,$ and $\sigma_0>0$, for $k=0,1,...$ until
convergence
\State compute the Cauchy point $s_k^C$
\If {condition \eqref{eq:solvercon} is satisfied}
\State compute a trial step  $\bar s_k$ using $\mathcal A_u$ (or $\mathcal A_r$) with an initial point $s_k^C$
\Else
\State compute a trial step $\bar s_k$ using $\mathcal A$ with an initial point $s_k^C$
\EndIf
\State set
    \begin{equation*}
      s_{k}=
      \left\{
        \begin{array}{ll}
        \bar s_k&\text{if}~m_k(\bar s_k)\le m_k(s_k^C)  \\
        s_k^C& \text{otherwise}
        \end{array}
      \right.
    \end{equation*}
\State compute $f(x_k+s_k)$ and
    \[\rho_k=\frac{f(x_k)-f(x_k+s_k)}{-m_k(s_k)}\]
\State set
    \begin{equation*}
      x_{k+1}=
      \left\{
        \begin{array}{ll}
        x_k+s_k &\text{if}~\rho_k\ge\eta_1  \\
        x_k & \text{otherwise}
        \end{array}
      \right.
    \end{equation*}
\State  set
    \begin{equation*}
      \sigma_{k+1}\in
      \left\{
        \begin{array}{lll}
        \left(0,\sigma_k\right] &\text{if}~\rho_k>\eta_2 &(\text{very successful iteration})  \\
        \left[\sigma_k,\gamma_1\sigma_k\right] &\text{if}~\eta_1\le\rho_k\le\eta_2 &(\text{successful iteration})\\
        \left[\gamma_1\sigma_k, \gamma_2\sigma_k\right] & \text{otherwise} &(\text{unsuccessful iteration})
        \end{array}
      \right.
    \end{equation*}
\end{algorithmic}
\end{algorithm}

We experimented with two subproblem solvers $\mathcal A_u$ for Algorithm~\ref{alg:arc}.
The first one is the gradient method with Barzilai-Borwein step size (\cite{barzilai1988two}) and the second one is NAG (here we denote it by APG to keep consistent with \cite{jiang2021accelerated}). More specifically, in our implementation,
if condition \eqref{eq:solvercon} is not satisfied, we still solve \eqref{opt:CRS} by BBM;
otherwise we implement BBM or APG to solve the unconstrained problem \eqref{eq:reform-u}.
The former  is termed \texttt{ARC-URBB}, while the latter is termed \texttt{ARC-URAPG}.
We compare our algorithms to the ARC algorithm in \cite{cartis2011adaptive}, denoted by \texttt{ARC-GLRT}, in which the subproblems are solved by the generalized Lanczos method.
Besides, we also implement Algorithm \ref{alg:arc} with two different subproblem solvers $\mathcal A_r$ in \cite{jiang2021accelerated}, denoted by \texttt{ARC-RBB} and \texttt{ARC-RAPG}, in which the subproblems are reformulated as \eqref{eq:reform} and solved by BBM and APG, respectively.

We implemented all the ARC algorithms in MATLAB R2017a on a Macbook Pro laptop with 4 Intel i5 cores (1.4GHz) and 8GB of RAM.
The implementations are based on 20 medium-size ($n\in[500,1500]$) problems from the CUTEst collections (\cite{gould2015cutest}) as in \cite{jiang2021accelerated}, where condition \eqref{eq:solvercon} is satisfied in at least one iteration in our new algorithm.
For condition~\eqref{eq:solvercon}, we set $\epsilon_1=10^{-2}$ and $\epsilon_2=10^{-4}$.
Other parameters in ARC are chosen as described in \cite{cartis2011adaptive}.
All the subproblem solvers use the same   eigenvalue tolerance, stopping criteria, and initialization as in \cite{jiang2021accelerated}.
For BBMs, a simple line search rule is used to guarantee the decrease of the objective function values.
For APGs,  a well known restarting strategy (\cite{o2015adaptive,ito2017unified}) is used to speed up the algorithm.

The numerical results are reported in Table~\ref{tab:cutest}.
The first column indicates the name of the problem instance with its  dimension.
The column $f^*$,  $n_i$, $n_{\text{prod}}$ , $n_f$ , $n_g$ and $n_{\text{eig}}$ show the final objective value, the iteration number, number of Hessian-vector products,
number of function evaluations, number of gradient evaluations and the number of eigenvalue computations.
The columns time, time$_\text{eig}$ and time$_\text{loop}$,
show in seconds the overall CPU time, eigenvalue computation time and difference between the last two, respectively.
Each value is an average of 10 realizations with different initial points. Table~\ref{tab:cutest} shows that with the same stopping criteria,
all algorithms return the same objective function value on 18 of the problems,
except  \texttt{ARC-RAPG}, \texttt{ARC-URBB} and \texttt{ARC-URAPG} on the problem BROYDN7D with a lower final objective function value,
and  \texttt{ARC-GLRT} on the problem CHAINWOO with a lower final objective function value.
 Table~\ref{tab:cutest} also shows the quantities $n_i$, $n_{\text{prod}}$, $n_f$ and $n_g$ of the five algorithms are similar. For several problems,
   \texttt{ARC-URBB} and \texttt{ARC-URAPG} based on our new reformulation have some advantages on $n_{\text{prod}}$.
Due to the eigenvalue calculation, four algorithms based on the convex reformulation require additional manipulation,
resulting in a larger total CPU time, evidenced by the column time, which was also observed in \cite{jiang2021accelerated}.
The column time$_\text{loop}$ shows that all the algorithms have a similar CPU time if we exclude the time for computing the eigenvalues.

To investigate the numerical results more clearly,
we illustrate the experiments by    performance profiles Figures \ref{fig:1}--\ref{fig:3} (\cite{dolan2002benchmarking}). According to the performs profiles, although \texttt{ARC-GLRT} has the best performance, the iteration numbers and the gradient evaluation numbers of \texttt{ARC-URBB} and \texttt{ARC-URAPG} are less than 2 times of those by \texttt{ARC-GLRT} on over 95\% of the tests,
and   Hessian-vector product number of \texttt{ARC-URBB} is less than 2 times of those by \texttt{ARC-GLRT} on about 85\% of the tests.
Noting that \texttt{ARC-URBB}, \texttt{ARC-URAPG}, \texttt{ARC-RBB} and \texttt{ARC-RAPG} have the similar performance,  we thus plot the performance profiles on test problems for these 4 algorithms in Figures \ref{fig:4}--\ref{fig:6}. We find \texttt{ARC-URAPG} has the best iteration number and gradient evaluation number, and both  \texttt{ARC-URBB} and \texttt{ARC-URAPG} have better Hessian-vector product number.

We also investigate the numerical results for all 10 implementations with different initializations,
in order to show the advantages of the new algorithms more comprehensively.
Table~\ref{tab:cutestfull} reports the number that \texttt{ARC-URBB} or \texttt{ARC-URAPG} outperforms \texttt{ARC-GLRT}, \texttt{ARC-RBB} and \texttt{ARC-RAPG} out of the 10 realizations for each problem.
It shows our algorithms frequently outperform \texttt{ARC-GLRT}, \texttt{ARC-RBB} and \texttt{ARC-RAPG} in  iteration number, number of
Hessian-vector products and gradient evaluations.

{
\centering
\scriptsize
\begin{longtable}{c|c|ccccccccc}
\toprule
Problem & Method
  &  $n_i$  & $n_{\text{prod}}$ & $n_f$ & $n_g$
  & $n_{\text{eig}}$ & $f^*$  & time  & time$_\text{eig}$
  &  time$_\text{loop}$  \\

\midrule
BROYDN7D &\texttt{\texttt{ARC-GLRT}}
&42.7 &828.9 &43.7 &35.6 &- &2.42e+02 &0.340 &- &0.340 \\
(1000) &\texttt{\texttt{ARC-RBB}}
&43.6 &946.9 &44.6 &36.3 &17.4 &2.40e+02 &1.071 &0.681 &0.390 \\
&\texttt{\texttt{ARC-RAPG}}
&43.7 &933.5 &44.7 &36.2 &17.9 &2.39e+02 &1.100 &0.701 &0.399 \\
&\texttt{\texttt{ARC-URBB}}
&43.8 &933.2 &44.8 &36.4 &17.9 &2.39e+02 &1.059 &0.688 &0.371 \\
&\texttt{\texttt{ARC-URAPG}}
&43.4 &845.8 &44.4 &35.9 &17.2 &2.39e+02 &1.058 &0.664 &0.394 \\

\midrule
BRYBND &\texttt{\texttt{ARC-GLRT}}
&34.3 &1575.5 &35.3 &29.3 &- &2.73e+01 &0.487 &- &0.487 \\
(1000) &\texttt{\texttt{ARC-RBB}}
&29.8 &1314.3 &30.8 &25.7 &6.5 &2.73e+01 &1.352 &0.946 &0.406 \\
&\texttt{\texttt{ARC-RAPG}}
&29.8 &1288.9 &30.8 &25.7 &6.5 &2.73e+01 &1.426 &1.005 &0.421 \\
&\texttt{\texttt{ARC-URBB}}
&29.8 &1278.7 &30.8 &25.7 &6.5 &2.73e+01 &1.349 &0.956 &0.394 \\
&\texttt{\texttt{ARC-URAPG}}
&29.8 &1262.1 &30.8 &25.7 &6.5 &2.73e+01 &1.376 &0.975 &0.402 \\

\midrule
CHAINWOO &\texttt{\texttt{ARC-GLRT}}
&203.5 &5462.5 &204.5 &152.7 &- &1.07e+03 &1.576 &- &1.576 \\
(1000) &\texttt{\texttt{ARC-RBB}}
&293.1 &10705.2 &294.1 &218.7 &172.5 &1.17e+03 &26.495 &23.215 &3.280 \\
&\texttt{\texttt{ARC-RAPG}}
&299.5 &10625.8 &300.5 &225.4 &178.4 &1.17e+03 &27.363 &23.841 &3.522 \\
&\texttt{\texttt{ARC-URBB}}
&291.4 &8363.3 &292.4 &219.4 &168.0 &1.17e+03 &26.276 &23.553 &2.723 \\
&\texttt{\texttt{ARC-URAPG}}
&303.4 &8683.5 &304.4 &227.9 &178.9 &1.16e+03 &26.534 &23.434 &3.099 \\

\midrule
DIXMAANF &\texttt{\texttt{ARC-GLRT}}
&23.8 &599.6 &24.8 &22.6 &- &1.00e+00 &0.421 &- &0.421 \\
(1500) &\texttt{\texttt{ARC-RBB}}
&22.1 &572.6 &23.1 &21.2 &10.1 &1.00e+00 &1.391 &0.964 &0.427 \\
&\texttt{\texttt{ARC-RAPG}}
&22.2 &543.8 &23.2 &21.1 &10.2 &1.00e+00 &1.391 &0.969 &0.423 \\
&\texttt{\texttt{ARC-URBB}}
&22.6 &535.3 &23.6 &21.5 &10.6 &1.00e+00 &1.415 &1.009 &0.406 \\
&\texttt{\texttt{ARC-URAPG}}
&22.3 &477.2 &23.3 &21.2 &10.3 &1.00e+00 &1.385 &0.974 &0.412 \\

\midrule
DIXMAANG &\texttt{\texttt{ARC-GLRT}}
&24.9 &606.7 &25.9 &23.0 &- &1.00e+00 &0.413 &- &0.413 \\
(1500) &\texttt{\texttt{ARC-RBB}}
&24.6 &652.8 &25.6 &22.6 &11.0 &1.00e+00 &1.446 &0.982 &0.464 \\
&\texttt{\texttt{ARC-RAPG}}
&23.7 &597.5 &24.7 &22.2 &10.1 &1.00e+00 &1.378 &0.927 &0.451 \\
&\texttt{\texttt{ARC-URBB}}
&23.0 &418.1 &24.0 &21.9 &9.8 &1.00e+00 &1.270 &0.912 &0.358 \\
&\texttt{\texttt{ARC-URAPG}}
&23.3 &441.9 &24.3 &22.0 &10.0 &1.00e+00 &1.274 &0.902 &0.372 \\

\midrule
DIXMAANH &\texttt{\texttt{ARC-GLRT}}
&29.6 &680.8 &30.6 &25.9 &- &1.00e+00 &0.461 &- &0.461 \\
(1500) &\texttt{\texttt{ARC-RBB}}
&30.7 &696.6 &31.7 &26.2 &13.3 &1.00e+00 &1.705 &1.186 &0.519 \\
&\texttt{\texttt{ARC-RAPG}}
&30.5 &664.4 &31.5 &26.1 &13.1 &1.00e+00 &1.696 &1.189 &0.507 \\
&\texttt{\texttt{ARC-URBB}}
&30.5 &625.3 &31.5 &26.0 &13.1 &1.00e+00 &1.659 &1.178 &0.480 \\
&\texttt{\texttt{ARC-URAPG}}
&30.5 &619.2 &31.5 &26.0 &13.4 &1.00e+00 &1.681 &1.198 &0.483 \\

\midrule
DIXMAANJ &\texttt{\texttt{ARC-GLRT}}
&43.7 &4519.5 &44.7 &37.6 &- &1.00e+00 &2.311 &- &2.311 \\
(1500) &\texttt{\texttt{ARC-RBB}}
&48.7 &2952.9 &49.7 &42.4 &30.3 &1.00e+00 &33.409 &31.727 &1.682 \\
&\texttt{\texttt{ARC-RAPG}}
&51.1 &3324.3 &52.1 &43.5 &33.1 &1.00e+00 &37.646 &35.873 &1.774 \\
&\texttt{\texttt{ARC-URBB}}
&50.1 &2937.5 &51.1 &43.5 &32.4 &1.00e+00 &35.312 &33.738 &1.574 \\
&\texttt{\texttt{ARC-URAPG}}
&49.2 &2743.4 &50.2 &42.7 &31.3 &1.00e+00 &33.913 &32.392 &1.521 \\

\midrule
DIXMAANK &\texttt{\texttt{ARC-GLRT}}
&51.1 &4883.9 &52.1 &43.2 &- &1.00e+00 &2.458 &- &2.458 \\
(1500) &\texttt{\texttt{ARC-RBB}}
&63.1 &4382.3 &64.1 &53.1 &42.9 &1.00e+00 &40.483 &38.208 &2.275 \\
&\texttt{\texttt{ARC-RAPG}}
&63.9 &4453.5 &64.9 &53.3 &43.8 &1.00e+00 &41.436 &39.114 &2.322 \\
&\texttt{\texttt{ARC-URBB}}
&60.7 &3523.5 &61.7 &51.5 &41.1 &1.00e+00 &39.471 &37.603 &1.868 \\
&\texttt{\texttt{ARC-URAPG}}
&62.4 &3962.9 &63.4 &52.5 &42.4 &1.00e+00 &40.038 &37.941 &2.097 \\

\midrule
DIXMAANL &\texttt{\texttt{ARC-GLRT}}
&57.7 &4569.5 &58.7 &47.6 &- &1.00e+00 &2.334 &- &2.334 \\
(1500) &\texttt{\texttt{ARC-RBB}}
&65.2 &4126.9 &66.2 &55.0 &40.9 &1.00e+00 &40.609 &38.454 &2.155 \\
&\texttt{\texttt{ARC-RAPG}}
&66.0 &4103.8 &67.0 &55.1 &42.1 &1.00e+00 &40.438 &38.255 &2.183 \\
&\texttt{\texttt{ARC-URBB}}
&61.3 &3398.6 &62.3 &52.2 &37.3 &1.00e+00 &37.372 &35.571 &1.801 \\
&\texttt{\texttt{ARC-URAPG}}
&65.4 &3721.6 &66.4 &54.6 &41.3 &1.00e+00 &39.836 &37.842 &1.995 \\

\midrule
EXTROSNB &\texttt{\texttt{ARC-GLRT}}
&1824.2 &54022.6 &1825.2 &1274.9 &- &1.47e-08 &16.641 &- &16.641 \\
(1000) &\texttt{\texttt{ARC-RBB}}
&1344.0 &192873.0 &1345.0 &1094.9 &1236.4 &2.99e-06 &72.713 &12.047 &60.666 \\
&\texttt{\texttt{ARC-RAPG}}
&1341.0 &192160.7 &1342.0 &1097.3 &1234.1 &2.99e-06 &71.442 &11.863 &59.579 \\
&\texttt{\texttt{ARC-URBB}}
&1383.9 &198543.3 &1384.9 &1129.2 &1276.2 &2.99e-06 &73.391 &12.177 &61.215 \\
&\texttt{\texttt{ARC-URAPG}}
&1397.9 &200543.8 &1398.9 &1121.8 &1291.3 &2.98e-06 &73.764 &12.293 &61.471 \\

\midrule
FLETCHCR &\texttt{\texttt{ARC-GLRT}}
&1969.9 &42563.3 &1970.9 &1327.2 &- &1.20e+00 &12.710 &- &12.710 \\
(1000) &\texttt{\texttt{ARC-RBB}}
&1982.0 &53970.1 &1983.0 &1357.0 &774.8 &1.20e+00 &38.237 &13.914 &24.324 \\
&\texttt{\texttt{ARC-RAPG}}
&1984.0 &53642.5 &1985.0 &1368.1 &787.8 &1.20e+00 &38.062 &13.562 &24.500 \\
&\texttt{\texttt{ARC-URBB}}
&1980.8 &52054.4 &1981.8 &1365.5 &782.0 &1.20e+00 &38.458 &14.402 &24.056 \\
&\texttt{\texttt{ARC-URAPG}}
&1976.8 &51214.6 &1977.8 &1361.9 &771.3 &1.20e+00 &38.390 &14.234 &24.156 \\

\midrule
FREUROTH &\texttt{\texttt{ARC-GLRT}}
&36.7 &366.1 &37.7 &30.3 &- &1.17e+05 &0.302 &- &0.302 \\
(1000) &\texttt{\texttt{ARC-RBB}}
&33.8 &1122.4 &34.8 &30.2 &21.2 &1.17e+05 &0.697 &0.205 &0.492 \\
&\texttt{\texttt{ARC-RAPG}}
&36.0 &1371.1 &37.0 &31.6 &23.5 &1.17e+05 &0.787 &0.222 &0.565 \\
&\texttt{\texttt{ARC-URBB}}
&36.5 &1400.4 &37.5 &30.2 &24.0 &1.17e+05 &0.885 &0.247 &0.639 \\
&\texttt{\texttt{ARC-URAPG}}
&34.8 &1207.6 &35.8 &30.1 &22.0 &1.17e+05 &0.810 &0.219 &0.591 \\

\midrule
GENHUMPS &\texttt{\texttt{ARC-GLRT}}
&1702.9 &50838.9 &1703.9 &1039.5 &- &8.73e-13 &15.912 &- &15.912 \\
(1000) &\texttt{\texttt{ARC-RBB}}
&1525.5 &41837.4 &1526.5 &922.5 &9.3 &7.06e-12 &20.876 &0.206 &20.670 \\
&\texttt{\texttt{ARC-RAPG}}
&1525.4 &41841.4 &1526.4 &922.4 &9.3 &8.90e-12 &19.249 &0.218 &19.030 \\
&\texttt{\texttt{ARC-URBB}}
&1525.5 &41729.6 &1526.5 &922.6 &9.3 &8.34e-12 &22.510 &0.199 &22.311 \\
&\texttt{\texttt{ARC-URAPG}}
&1525.4 &41762.7 &1526.4 &922.5 &9.3 &1.44e-11 &21.270 &0.200 &21.071 \\

\midrule
GENROSE &\texttt{\texttt{ARC-GLRT}}
&1058.6 &20703.8 &1059.6 &711.7 &- &1.00e+00 &2.818 &- &2.818 \\
(500) &\texttt{\texttt{ARC-RBB}}
&1079.7 &28236.6 &1080.7 &736.9 &166.5 &1.00e+00 &4.092 &0.944 &3.149 \\
&\texttt{\texttt{ARC-RAPG}}
&1151.5 &29887.7 &1152.5 &780.7 &191.6 &1.00e+00 &3.594 &0.863 &2.732 \\
&\texttt{\texttt{ARC-URBB}}
&1081.1 &28124.1 &1082.1 &737.5 &164.3 &1.00e+00 &3.848 &0.890 &2.958 \\
&\texttt{\texttt{ARC-URAPG}}
&1083.4 &27979.1 &1084.4 &737.1 &169.7 &1.00e+00 &4.268 &1.005 &3.263 \\

\midrule
NONCVXU2 &\texttt{\texttt{ARC-GLRT}}
&65.5 &8065.7 &66.5 &61.5 &- &2.32e+03 &2.083 &- &2.083 \\
(1000) &\texttt{\texttt{ARC-RBB}}
&127.5 &7660.5 &128.5 &122.1 &124.5 &2.32e+03 &78.082 &75.564 &2.518 \\
&\texttt{\texttt{ARC-RAPG}}
&122.4 &7637.8 &123.4 &118.9 &119.6 &2.32e+03 &77.664 &75.163 &2.501 \\
&\texttt{\texttt{ARC-URBB}}
&123.4 &7845.9 &124.4 &119.6 &120.6 &2.32e+03 &78.858 &76.286 &2.571 \\
&\texttt{\texttt{ARC-URAPG}}
&113.8 &7211.2 &114.8 &109.9 &111.0 &2.32e+03 &71.320 &69.111 &2.209 \\

\midrule
NONCVXUN &\texttt{\texttt{ARC-GLRT}}
&300.9 &224970.5 &301.9 &294.5 &- &2.32e+03 &49.239 &- &49.239 \\
(1000) &\texttt{\texttt{ARC-RBB}}
&2025.5 &283403.7 &2026.5 &2018.8 &2021.6 &2.32e+03 &1414.201 &1346.432 &67.769 \\
&\texttt{\texttt{ARC-RAPG}}
&2116.6 &295837.9 &2117.6 &2109.8 &2112.9 &2.32e+03 &1479.271 &1407.690 &71.581 \\
&\texttt{\texttt{ARC-URBB}}
&2483.0 &350646.4 &2484.0 &2477.1 &2479.3 &2.32e+03 &1734.238 &1650.711 &83.527 \\
&\texttt{\texttt{ARC-URAPG}}
&2105.2 &294676.3 &2106.2 &2098.3 &2101.3 &2.32e+03 &1466.452 &1393.276 &73.176 \\

\midrule
OSCIPATH &\texttt{\texttt{ARC-GLRT}}
&39.3 &6079.9 &40.3 &31.3 &- &3.12e-01 &0.516 &- &0.516 \\
(500) &\texttt{\texttt{ARC-RBB}}
&56.4 &5658.1 &57.4 &49.5 &27.8 &3.12e-01 &5.502 &5.204 &0.298 \\
&\texttt{\texttt{ARC-RAPG}}
&57.0 &5747.5 &58.0 &49.6 &28.1 &3.12e-01 &5.702 &5.348 &0.354 \\
&\texttt{\texttt{ARC-URBB}}
&58.7 &6007.2 &59.7 &52.0 &30.5 &3.12e-01 &6.163 &5.845 &0.318 \\
&\texttt{\texttt{ARC-URAPG}}
&57.3 &5799.2 &58.3 &50.8 &28.7 &3.12e-01 &5.866 &5.532 &0.334 \\

\midrule
TOINTGSS &\texttt{\texttt{ARC-GLRT}}
&19.2 &118.6 &20.2 &14.1 &- &1.00e+01 &0.119 &- &0.119 \\
(1000) &\texttt{\texttt{ARC-RBB}}
&15.4 &368.2 &16.4 &12.2 &10.3 &1.00e+01 &0.269 &0.086 &0.183 \\
&\texttt{\texttt{ARC-RAPG}}
&15.9 &494.7 &16.9 &12.4 &10.9 &1.00e+01 &0.296 &0.081 &0.215 \\
&\texttt{\texttt{ARC-URBB}}
&15.0 &322.9 &16.0 &11.8 &10.0 &1.00e+01 &0.233 &0.076 &0.156 \\
&\texttt{\texttt{ARC-URAPG}}
&15.6 &372.6 &16.6 &12.5 &10.6 &1.00e+01 &0.255 &0.082 &0.173 \\

\midrule
TQUARTIC &\texttt{\texttt{ARC-GLRT}}
&63.9 &282.1 &64.9 &52.9 &- &2.37e-14 &0.363 &- &0.363 \\
(1000) &\texttt{\texttt{ARC-RBB}}
&71.5 &838.6 &72.5 &55.6 &6.9 &1.99e-13 &0.598 &0.084 &0.513 \\
&\texttt{\texttt{ARC-RAPG}}
&71.0 &934.8 &72.0 &55.4 &6.6 &8.58e-11 &0.674 &0.090 &0.584 \\
&\texttt{\texttt{ARC-URBB}}
&71.3 &566.7 &72.3 &55.6 &6.9 &1.04e-10 &0.521 &0.091 &0.431 \\
&\texttt{\texttt{ARC-URAPG}}
&72.3 &926.4 &73.3 &56.5 &6.7 &3.79e-10 &0.639 &0.087 &0.552 \\

\midrule
WOODS &\texttt{\texttt{ARC-GLRT}}
&286.4 &4542.6 &287.4 &210.2 &- &8.66e-15 &1.561 &- &1.561 \\
(1000) &\texttt{\texttt{ARC-RBB}}
&382.8 &9574.8 &383.8 &264.5 &6.2 &1.88e-12 &3.733 &0.067 &3.666 \\
&\texttt{\texttt{ARC-RAPG}}
&381.3 &9426.5 &382.3 &263.9 &5.5 &3.15e-14 &3.340 &0.051 &3.288 \\
&\texttt{\texttt{ARC-URBB}}
&382.2 &9486.3 &383.2 &264.2 &6.1 &1.67e-14 &3.722 &0.067 &3.655 \\
&\texttt{\texttt{ARC-URAPG}}
&381.7 &9542.6 &382.7 &264.6 &6.3 &1.67e-12 &3.833 &0.068 &3.765 \\
\bottomrule
\caption{\small Results on the CUTEst problems}
\label{tab:cutest}
\end{longtable}
}
{\centering
\scriptsize
\begin{longtable}{c|c|ccc|ccc}

  \toprule
  \multirow{2}{*}{Problem}
  &\multirow{2}{*}{Index}
  &\multicolumn{3}{c|}{\texttt{ARC-URBB}}
  &\multicolumn{3}{c}{\texttt{ARC-URAPG}} \\
  \cmidrule(r){3-5} \cmidrule(r){6-8}
  &&  \texttt{ARC-GLRT}  & \texttt{ARC-RBB} & \texttt{ARC-RAPG}
  &  \texttt{ARC-GLRT} & \texttt{ARC-RBB} & \texttt{ARC-RAPG} \\

  \midrule
  \multirow{3}{*}{BROYDN7D}
  &$n_i$ &4 &5 &2 &4 &6 &4 \\
  &$n_{\text{prod}}$ &4 &5 &6 &5 &8 &6 \\
  &$n_g$ &3 &3 &1 &4 &4 &2 \\

  \midrule
  \multirow{3}{*}{BRYBND}
  &$n_i$ &10 &0 &0 &10 &0 &0 \\
  &$n_{\text{prod}}$ &7 &7 &4 &7 &9 &6 \\
  &$n_g$ &10 &0 &0 &10 &0 &0 \\

  \midrule
  \multirow{3}{*}{CHAINWOO}
  &$n_i$ &0 &5 &7 &0 &3 &5 \\
  &$n_{\text{prod}}$ &0 &10 &9 &0 &10 &9 \\
  &$n_g$ &0 &6 &6 &0 &2 &5 \\

  \midrule
  \multirow{3}{*}{DIXMAANF}
  &$n_i$ &7 &0 &2 &6 &1 &0 \\
  &$n_{\text{prod}}$ &5 &6 &5 &7 &6 &7 \\
  &$n_g$ &6 &0 &1 &6 &2 &0 \\

  \midrule
  \multirow{3}{*}{DIXMAANG}
  &$n_i$ &7 &5 &4 &6 &6 &3 \\
  &$n_{\text{prod}}$ &10 &10 &9 &9 &10 &9 \\
  &$n_g$ &6 &4 &4 &7 &6 &2 \\

  \midrule
  \multirow{3}{*}{DIXMAANH}
  &$n_i$ &3 &2 &2 &4 &2 &1 \\
  &$n_{\text{prod}}$ &7 &6 &5 &6 &7 &6 \\
  &$n_g$ &4 &2 &2 &3 &2 &1 \\

  \midrule
  \multirow{3}{*}{DIXMAANJ}
  &$n_i$ &1 &5 &5 &1 &5 &4 \\
  &$n_{\text{prod}}$ &10 &7 &7 &10 &9 &7 \\
  &$n_g$ &1 &6 &5 &1 &3 &3 \\

  \midrule
  \multirow{3}{*}{DIXMAANK}
  &$n_i$ &0 &7 &7 &1 &5 &4 \\
  &$n_{\text{prod}}$ &9 &7 &9 &8 &8 &7 \\
  &$n_g$ &0 &5 &7 &0 &5 &5 \\

  \midrule
  \multirow{3}{*}{DIXMAANL}
  &$n_i$ &4 &7 &6 &2 &4 &5 \\
  &$n_{\text{prod}}$ &8 &8 &8 &9 &7 &7 \\
  &$n_g$ &2 &6 &6 &1 &3 &5 \\

  \midrule
  \multirow{3}{*}{EXTROSNB}
  &$n_i$ &10 &3 &3 &10 &2 &3 \\
  &$n_{\text{prod}}$ &0 &3 &3 &0 &2 &3 \\
  &$n_g$ &10 &3 &3 &10 &5 &5 \\

  \midrule
  \multirow{3}{*}{FLETCHCR}
  &$n_i$ &5 &4 &5 &5 &5 &6 \\
  &$n_{\text{prod}}$ &4 &8 &8 &5 &9 &8 \\
  &$n_g$ &2 &2 &6 &2 &2 &9 \\

  \midrule
  \multirow{3}{*}{FREUROTH}
  &$n_i$ &5 &4 &4 &7 &3 &5 \\
  &$n_{\text{prod}}$ &0 &5 &7 &0 &4 &6 \\
  &$n_g$ &5 &5 &6 &4 &4 &7 \\

  \midrule
  \multirow{3}{*}{GENHUMPS}
  &$n_i$ &10 &1 &0 &10 &2 &0 \\
  &$n_{\text{prod}}$ &10 &7 &8 &10 &7 &7 \\
  &$n_g$ &10 &0 &0 &10 &1 &0 \\

  \midrule
  \multirow{3}{*}{GENROSE}
  &$n_i$ &3 &5 &4 &3 &4 &4 \\
  &$n_{\text{prod}}$ &1 &5 &5 &1 &6 &5 \\
  &$n_g$ &3 &4 &6 &3 &5 &5 \\

  \midrule
  \multirow{3}{*}{NONCVXU2}
  &$n_i$ &0 &7 &6 &0 &9 &7 \\
  &$n_{\text{prod}}$ &6 &4 &6 &7 &8 &6 \\
  &$n_g$ &0 &5 &6 &0 &8 &7 \\

  \midrule
  \multirow{3}{*}{NONCVXUN}
  &$n_i$ &0 &4 &4 &0 &7 &5 \\
  &$n_{\text{prod}}$ &1 &4 &4 &2 &7 &5 \\
  &$n_g$ &0 &4 &4 &0 &7 &5 \\

  \midrule
  \multirow{3}{*}{OSCIPATH}
  &$n_i$ &0 &4 &3 &0 &4 &5 \\
  &$n_{\text{prod}}$ &5 &4 &3 &5 &4 &5 \\
  &$n_g$ &0 &3 &3 &0 &5 &4 \\

  \midrule
  \multirow{3}{*}{TOINTGSS}
  &$n_i$ &6 &1 &3 &7 &1 &1 \\
  &$n_{\text{prod}}$ &1 &5 &9 &0 &6 &9 \\
  &$n_g$ &7 &2 &4 &7 &2 &1 \\

  \midrule
  \multirow{3}{*}{TQUARTIC}
  &$n_i$ &4 &3 &4 &3 &4 &4 \\
  &$n_{\text{prod}}$ &0 &9 &10 &0 &4 &7 \\
  &$n_g$ &5 &3 &4 &3 &4 &4 \\

  \midrule
  \multirow{3}{*}{WOODS}
  &$n_i$ &0 &6 &1 &0 &4 &4 \\
  &$n_{\text{prod}}$ &0 &6 &4 &0 &6 &3 \\
  &$n_g$ &0 &6 &1 &0 &4 &2 \\

  \bottomrule

  \caption{\small
  The number of times \texttt{ARC-URBB} (or \texttt{\texttt{ARC-URAPG}}) performs better than the other three algorithms in 10 realizations.
  }
  \label{tab:cutestfull}
\end{longtable}
}

\begin{figure}[!h]
    \centering
    \includegraphics[width=9cm]{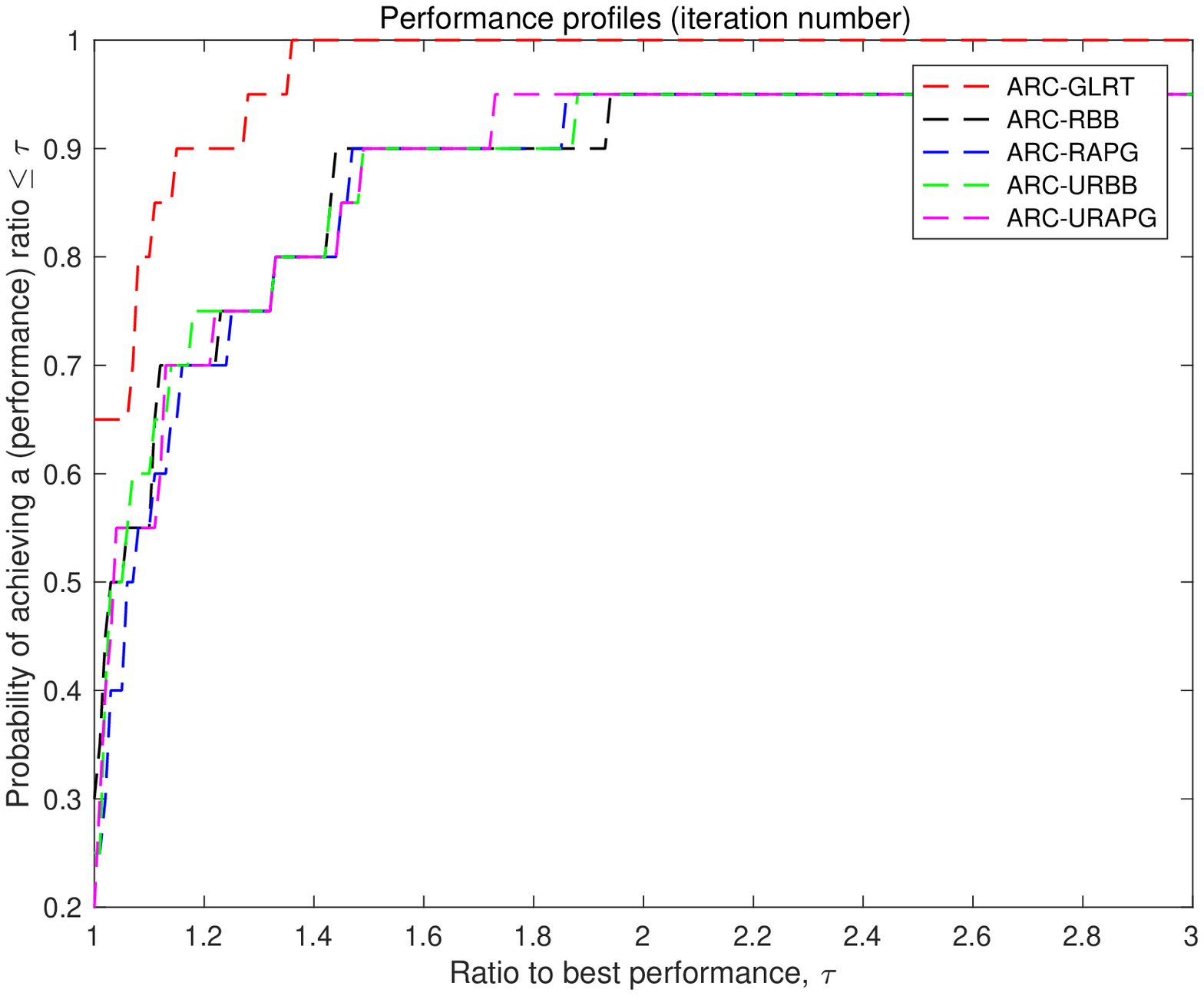}
    \caption{Performance profiles for iteration number for \texttt{ARC-GLRT}, \texttt{ARC-RBB}, \texttt{ARC-RAPG}, \texttt{ARC-URBB} and \texttt{ARC-URAPG} on the CUTEst problems}
    \label{fig:1}
  \end{figure}
  \begin{figure}[!h]
    \centering
    \includegraphics[width=9cm]{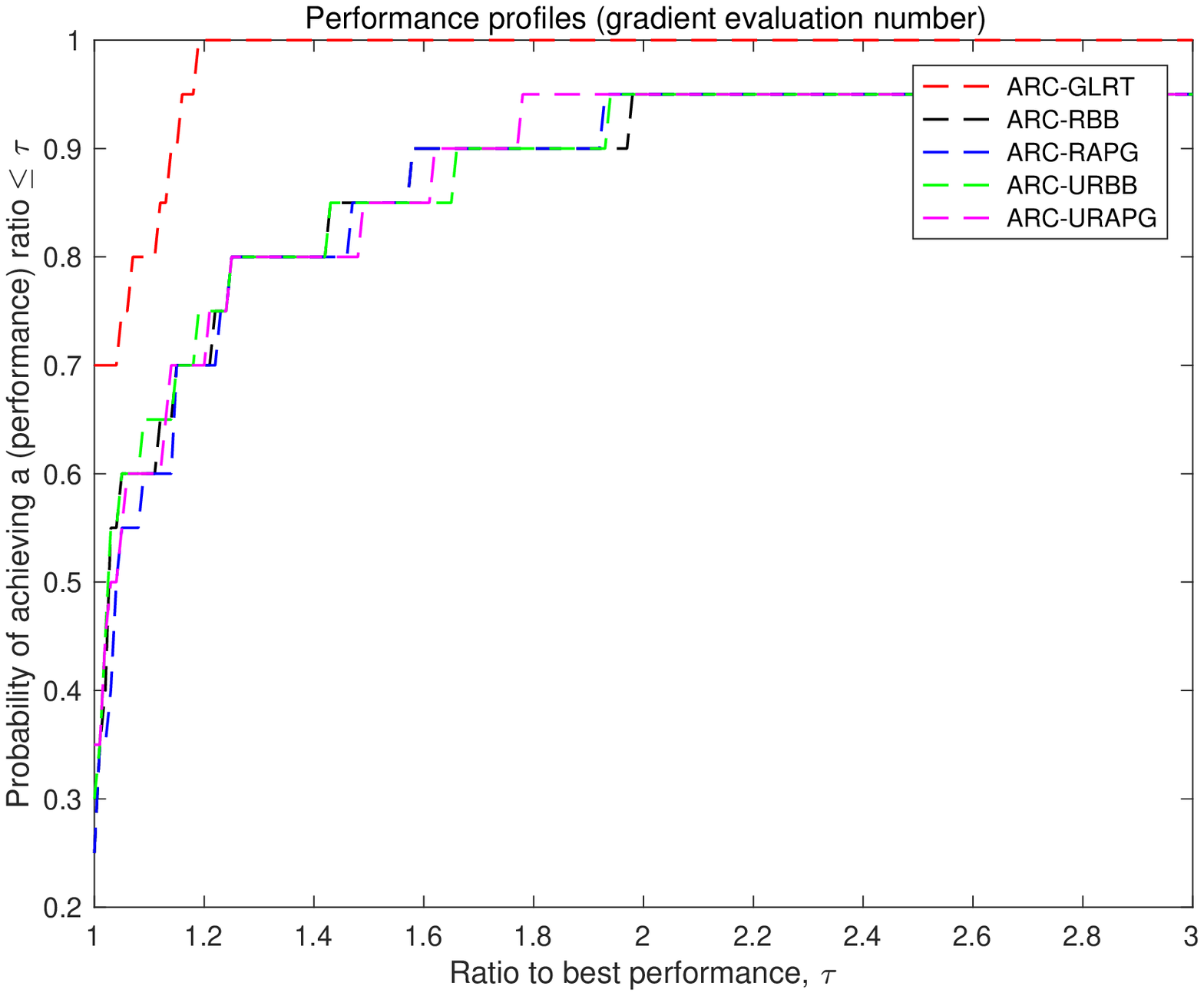}
    \caption{Performance profiles for gradient evaluations for \texttt{ARC-GLRT}, \texttt{ARC-RBB}, \texttt{ARC-RAPG}, \texttt{ARC-URBB} and \texttt{ARC-URAPG} on the CUTEst problems}
    \label{fig:2}
  \end{figure}
  \begin{figure}[!h]
    \centering
    \includegraphics[width=9cm]{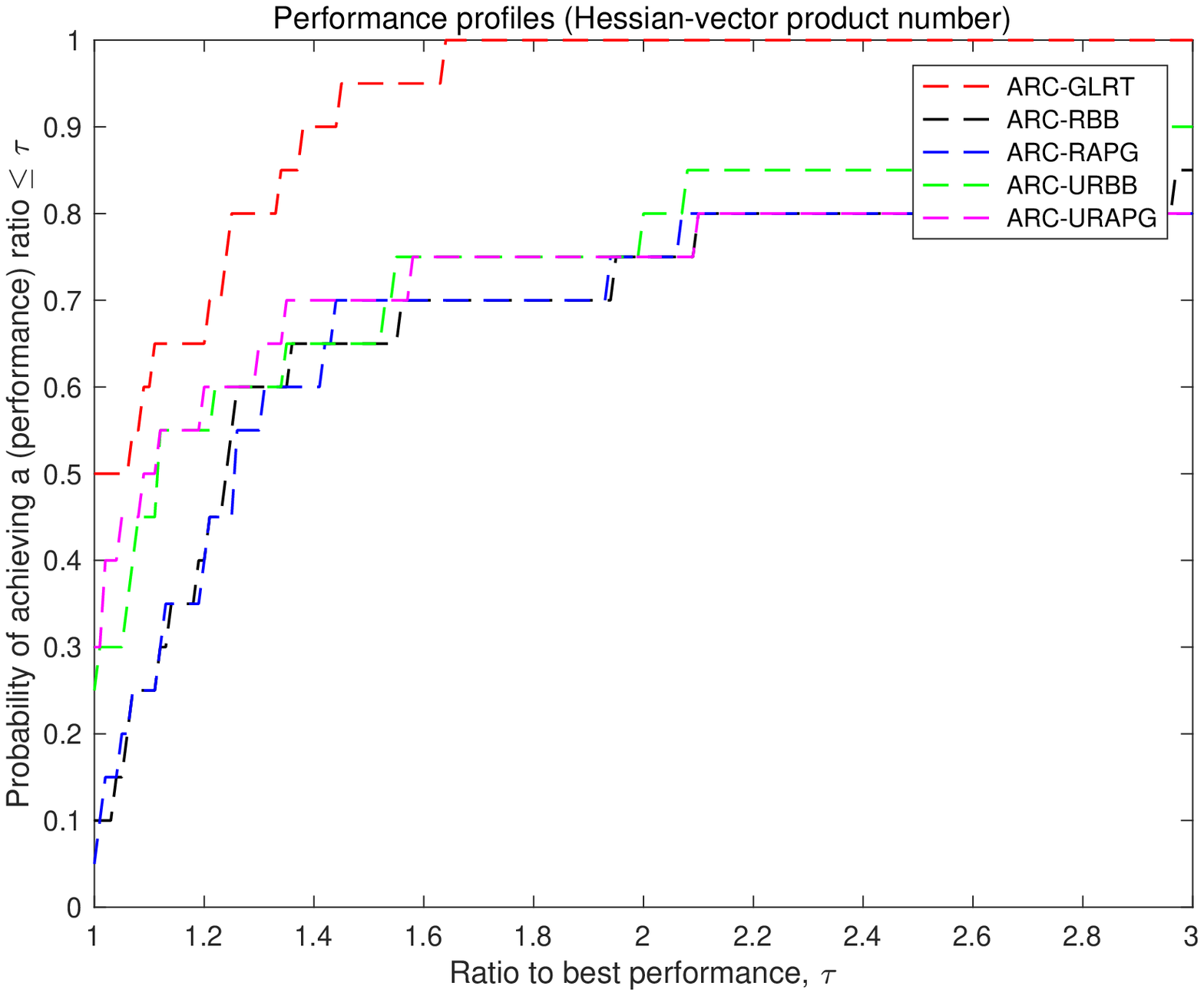}
    \caption{Performance profiles for Hessian-vector products for \texttt{ARC-GLRT}, \texttt{ARC-RBB}, \texttt{ARC-RAPG}, \texttt{ARC-URBB} and \texttt{ARC-URAPG} on the CUTEst problems}
    \label{fig:3}
  \end{figure}
  \begin{figure}[!h]
    \centering
    \includegraphics[width=9cm]{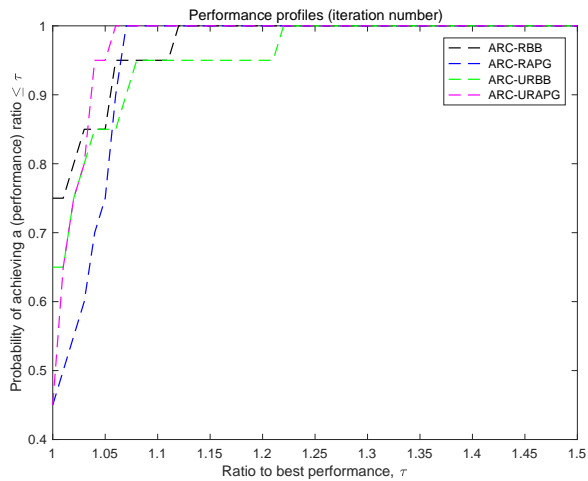}
    \caption{Performance profiles for iteration number for \texttt{ARC-RBB}, \texttt{ARC-RAPG}, \texttt{ARC-URBB} and \texttt{ARC-URAPG} on the CUTEst problems}
    \label{fig:4}
  \end{figure}
  \begin{figure}[!h]
    \centering
    \includegraphics[width=9cm]{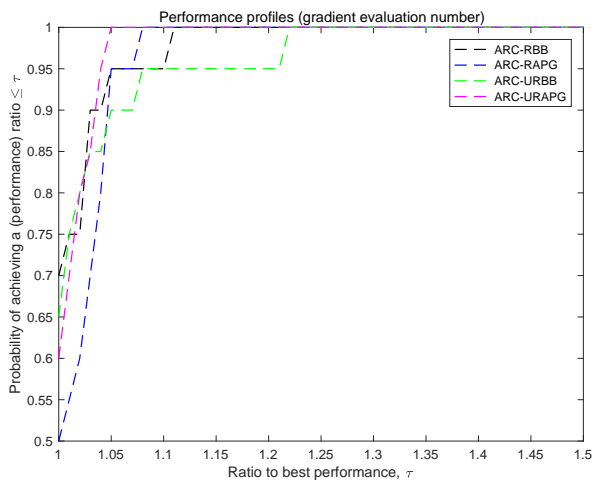}
    \caption{Performance profiles for gradient evaluation for \texttt{ARC-RBB}, \texttt{ARC-RAPG}, \texttt{ARC-URBB} and \texttt{ARC-URAPG} on the CUTEst problems}
    \label{fig:5}
  \end{figure}
  \begin{figure}[!h]
    \centering
    \includegraphics[width=9cm]{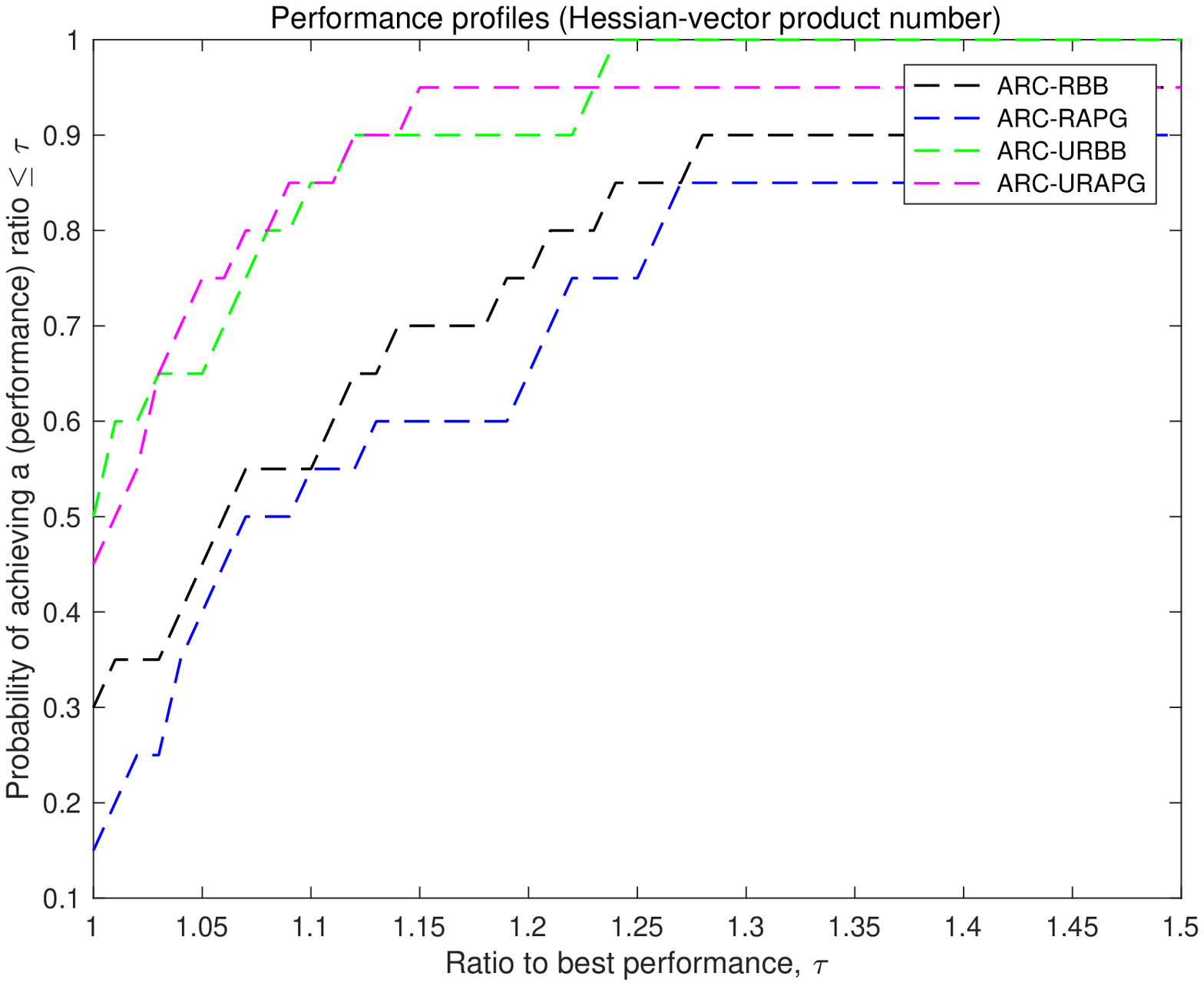}
    \caption{Performance profiles for Hessian-vector products for \texttt{ARC-RBB}, \texttt{ARC-RAPG}, \texttt{ARC-URBB} and \texttt{ARC-URAPG} on the CUTEst problems}
    \label{fig:6}
  \end{figure}

\section{Conclusion}\label{sec:6}
In this paper, we propose a new convex reformulation for the subproblem of the cubic regularization methods.
Based on our reformulation, we propose a variant of the non-adaptive CR algorithm that admits an iteration complexity $\mO(\epsilon_g^{-3/2})$ to find an $(\epsilon_g,\sqrt{L\epsilon_g})$ stationary point.
Moreover, we show that an operation complexity bound of our algorithm is $\tO(\epsilon_g^{-7/4})$ when the subproblems are solved by Nesterov's accelerated gradient method and the approximated eigenvalues are computed by the Lanczos procedure.
We also propose a variant of the ARC algorithm with similar complexity guarantees.
Both of our iteration and operation complexity bounds match the best known bounds in the literature for algorithms that based on first- and second-order information.
Numerical experiments on the ARC equipped with our reformulation for solving subproblems also illustrate the effectiveness of our approach.

For future research, we would like to explore if our reformulation can be extended to solve auxiliary problems in tensor methods for
unconstrained optimization (\cite{birgin2017worst,jiang2020unified,nesterov2021implementable,grapiglia2021inexact}), which were shown to have fast global convergence guarantees.
It is well known that  the auxiliary problem in the model function in each iteration of the tensor method is a regularized $p$-th order Taylor approximation, which is difficult to solve.
Two recent works \cite{nesterov2021implementable,grapiglia2021inexact}
show that for $p=3$ and   convex minimization problems, the Tensor model can be solved by an adaptive Bregman proximal gradient method, where each subproblem is of form
\[
\min_x c^{\top}x+\frac{1}{2}x^{\top}Ax+\frac{\gamma}{3+\nu}\|x\|^{3+\nu}, \text{ with  } A\succeq0.\]
It will be interesting to see if the methods in  \cite{nesterov2021implementable,grapiglia2021inexact}
  can be extended to nonconvex minimization problems and still have similar subproblems, and if  our reformulation can be extended to solving these subproblems.

\begin{acknowledgements}
This paper is dedicated to the memory of Professor Duan Li.
The first and third authors would like to thank Professor Duan Li, for his advice, help and encouragement during their Ph.D and postdoctoral time in the Chinese University of Hong Kong. The authors would like to thank the two anonymous referees for the invaluable comments that improve the quality of the paper significantly.
The first author is supported in part by NSFC 11801087 and 12171100.
\end{acknowledgements}

\section{Appendix}
\subsection{Proofs for Lemma \ref{lem:bd1}}
In  this case, $s_k$ is approximately computed by \eqref{eq:mkr} or \eqref{eq:amkr}.
Let $\sigma=L/2$ for Algorithm \ref{alg:cr} and $\sigma=\sigma_k$ for Algorithm \ref{alg:ourarc}. Noting that $d_k=s_k$, by \eqref{eq:gmkr}, we have
\begin{equation}
\label{eq:case2-1-opt-le}
\left\|g_k + H_kd_k +3\epsilon_E d_k + \sigma\|d_k\|d_k\right\| \le \epsilon_S.
\end{equation}
Since
\begin{equation*}
\label{eq:case2-1-opt-ge}
\left\|g_k + H_kd_k +3\epsilon_E d_k + \sigma\|d_k\|d_k\right\| \ge \left\|g_k + H_kd_k\right\|-3\epsilon_E \|d_k\|-\sigma\|d_k\|^2,
\end{equation*}
we have
\begin{equation}
\label{eq:case2-1-g-Hd}
\left\|g_k + H_kd_k\right\|\le \epsilon_S+3\epsilon_E \|d_k\|+\sigma\|d_k\|^2.
\end{equation}
Due to $\alpha_k\le\lambda_{\text{min}}(H_k)+\epsilon_E$ and $\alpha_k\ge-\epsilon_E$, we have
\begin{equation}
\label{eq:case2-1-lambda}
\lambda_{\text{min}}(H_k)\ge -2\epsilon_E.
\end{equation}
Using \eqref{eq:case2-1-opt-le}, we have
\begin{equation}
\label{eq:case2-1-gd-dHd}
\begin{aligned}
g_k^{\top}d_k + d_k^{\top}H_kd_k +3\epsilon_E \|d_k\|^2+ \sigma\|d_k\|^3 &\le \left\|g_k + H_kd_k +3\epsilon_E d_k + \sigma\|d_k\|d_k\right\|\cdot\|d_k\|\\
& \le \epsilon_S\|d_k\|.
\end{aligned}
\end{equation}
Then  we have
\begin{equation}
\label{eq:case2-1-gd}
\begin{aligned}
g_k^{\top}d_k &\overset{ \eqref{eq:case2-1-gd-dHd}}{\le}\epsilon_S\|d_k\| - d_k^{\top}H_kd_k - 3\epsilon_E \|d_k\|^2 - \sigma\|d_k\|^3\\
& = \epsilon_S\|d_k\| - (d_k^{\top}H_kd_k + 3\epsilon_E \|d_k\|^2) - \sigma\|d_k\|^3\\
& \overset{\eqref{eq:case2-1-lambda}}{\le}\epsilon_S\|d_k\| -\epsilon_E \|d_k\|^2- \sigma\|d_k\|^3,
\end{aligned}
\end{equation}
By \eqref{eq:case2-1-gd-dHd}, we also have
\begin{equation}
\label{eq:case2-1-dHd}
d_k^{\top}H_kd_k \le -g_k^{\top}d_k - 3\epsilon_E \|d_k\|^2 - \sigma\|d_k\|^3 + \epsilon_S\|d_k\|.
\end{equation}
Hence we obtain
\begin{equation*}
\label{eq:case2-1-fun-le}
\begin{aligned}
m(d_{k}) &= g_k^{\top}d_k + \frac{1}{2}d_k^{\top}H_kd_k + \frac{\sigma}{3}\|d_k\|^3 \\
&\overset{ \eqref{eq:case2-1-dHd}}{ \le} g_k^{\top}d_k - \frac{1}{2}g_k^{\top}d_k -\frac{3}{2} \epsilon_E\|d_k\|^2 - \frac{\sigma}{2}\|d_k\|^3 + \frac{1}{2}\epsilon_S\|d_k\| + \frac{\sigma}{3}\|d_k\|^3 \\
& = \frac{1}{2}g_k^{\top}d_k - \frac{3}{2}\epsilon_E\|d_k\|^2 + \frac{1}{2}\epsilon_S\|d_k\| - \frac{\sigma}{6}\|d_k\|^3 \\
&\overset{ \eqref{eq:case2-1-gd}}{\le} \frac{1}{2}\epsilon_S\|d_k\| - \frac{\sigma}{2}\|d_k\|^3 -2 \epsilon_E\|d_k\|^2 + \frac{1}{2}\epsilon_S\|d_k\| - \frac{\sigma}{6}\|d_k\|^3 \\
& = - \frac{2\sigma}{3}\|d_k\|^3 - 2\epsilon_E\|d_k\|^2 + \epsilon_S\|d_k\|,
\end{aligned}
\end{equation*}
or equivalently,
\begin{equation*}
\label{eq:case2-1-fun-ge}
-m(d_{k})\ge \frac{2\sigma}{3}\|d_k\|^3 + 2\epsilon_E\|d_k\|^2 - \epsilon_S\|d_k\|.
\end{equation*}
Due to $\epsilon_S=\epsilon_E^2/L$ as in Condition \ref{con:eps}, it follows  that
\begin{equation}
\label{eq:case1-mdk}
-m(d_{k})\ge\frac{2\sigma}{3}\|d_k\|^3 +2 \epsilon_E\|d_k\|^2 - \frac{\epsilon_E^2}{L}\|d_k\|.
\end{equation}

Now suppose that $x_k+d_k$ is not an $(\epsilon_g,\sqrt{L\epsilon_g})$  stationary point. We then have either (i) $\|\nabla f(x_k+d_k)\|>\epsilon_g$, or (ii) $\|\nabla f(x_k+d_k)\|\le \epsilon_g,~\lambda_{\min}(\nabla^2 f(x_k+d_k)) < -\sqrt{L\epsilon_g}$. Let us consider the following cases (i) and (ii) separately.
\begin{itemize}
\item[(i)]  Using \eqref{eq:case2-1-g-Hd}, Taylor expansion for $\nabla f(x_k+d_k)$ and \eqref{eq:hlip}, we have
\begin{equation}
\label{eq:case2-1-g-next-le}
\begin{array}{rcl}
\|\nabla f(x_k+d_k)\|& \le& \|g_k + H_kd_k\| + \frac{L}{2}\|d_k\|^2\\
&\le&\epsilon_S +3\epsilon_E\|d_k\| + (\sigma+\frac{L}{2})\|d_k\|^2,
\end{array}
\end{equation}
where the first inequality follows from Lemma 1 in \cite{nesterov2006cubic}.
Using $\|\nabla f(x_k+d_k)\|>\epsilon_g$ and \eqref{eq:case2-1-g-next-le}, we obtain
\begin{equation*}
\label{eq:case2-1-epsg}
\epsilon_g \leq \epsilon_S+3\epsilon_E\|d_k\| + (\sigma+\frac{L}{2})\|d_k\|^2.
\end{equation*}
This gives
$$\|d_{k}\|\ge\frac{-3\epsilon_E+\sqrt{9\epsilon_E^2-4(\sigma+\frac{L}{2})(\epsilon_S-\epsilon_g)}}{2\sigma+L}
\quad\text{or}\quad
\|d_{k}\|\le\frac{-3\epsilon_E-\sqrt{9\epsilon_E^2-4(\sigma+\frac{L}{2})(\epsilon_S-\epsilon_g)}}{2\sigma+L}.
$$
where the second case is discarded since $\|d_{k}\|\ge 0$. Due to $\epsilon_S=\epsilon_E^2/L$ and $\epsilon_g=\epsilon_E^2/L$  in Condition \ref{con:eps}, we further have
\begin{equation}\label{eq:case1-dkor}
        \|d_{k}\| \ge \frac{-3\epsilon_E+\sqrt{25\epsilon_E^2+(32\sigma/L)\epsilon_E^2}}{2\sigma+L}.
\end{equation}

\begin{itemize}
\item For Algorithm \ref{alg:cr}, due to $\sigma=L/2$, \eqref{eq:case1-dkor} yields
\begin{equation*}
\|d_{k}\| \ge \frac{3}{2L}\epsilon_E,
\end{equation*}
Therefore, using $\sigma=L/2$, we have
\begin{equation*}
-m(d_{k})\overset{\eqref{eq:case1-mdk}}{\ge} \frac{33\epsilon_E^3}{8L^2}.
\end{equation*}
\item
Note that from Lemma \ref{lem:gamma}, we have $\sigma=\sigma_k\le\max\{\sigma_0, \frac{\gamma L}{2}\}$ for Algorithm \ref{alg:ourarc}.
Therefore if $\sigma\ge L/2$, \eqref{eq:case1-dkor} yields
\begin{equation}
\label{eq:case1-dkbd1}
\|d_{k}\| \ge \frac{3}{(\max\{\gamma,2\sigma_0/L\}+1)L}\epsilon_E,
\end{equation}
and if $\sigma<L/2$, \eqref{eq:case1-dkor} yields
\begin{equation}
\label{eq:case1-dkbd2}
\|d_{k}\| \ge \frac{1}{L}\epsilon_E,
\end{equation}

Using \eqref{eq:case1-dkor}, we claim the following inequality holds,
\begin{equation}
\label{eq:keyeq}
\frac{2\sigma}{3}\|d_k\|^2 +2 \epsilon_E\|d_k\| - \frac{\epsilon_E^2}{L}\ge \frac{3\epsilon_E^2}{L},
\end{equation}
provided $\sigma>0$.
Now combining \eqref{eq:case1-mdk}, \eqref{eq:case1-dkbd1}, \eqref{eq:case1-dkbd2} and \eqref{eq:keyeq}, we have
\[
-m(d_k)\ge \frac{3\epsilon_E^2}{L}\|d_k\|\ge \min\left\{\frac{3}{\max\{\gamma,2\sigma_0/L\}+1},1\right\}\cdot\frac{3\epsilon_E^3}{L^2}
\]
Indeed, to prove inequality \eqref{eq:keyeq}, we only need to show, using \eqref{eq:case1-dkor},
\begin{equation*}
\frac{2(\sqrt{25+32a}-3)^2}{3(4a+4+1/a)}+\frac{2(\sqrt{25+32a}-3)}{2a+1}=\frac{2}{3}\frac{32a^2+16a+3\sqrt{25+32a}-9}{(2a+1)^2}\ge 4,
\end{equation*}
for $a=\sigma/L>0$. Note that
the above inequality is equivalent to
\[
\psi(a)=8a^2-8a+3\sqrt{25+32a}-15\ge0.
\]
As $\psi(0)=0$, it suffices to show
\[
\psi'(a)=16a-8+\frac{48}{\sqrt{25+32a}}>0,~\forall a\ge0.
\]
This holds because
\[
\psi''(a)=16-48\frac{16}{(\sqrt{25+32a})^{3}}
=16(1-\frac{48}{(\sqrt{25+32a})^3})\ge 16(1-\frac{48}{125})>0,~\forall a\ge0,
\]
and $\psi'(0)=\frac{8}{5}\ge0$.
\end{itemize}

\item[(ii)] By Assumption \ref{asmp:1} and \eqref{eq:case2-1-lambda}, we have
\begin{equation}
\label{eq:case2-1-H-ineq}
\lambda_{\min}(\nabla^2 f(x_k+d_k)) \geq \lambda_{\min}(H_k) - L\|d_k\| \geq -2\epsilon_E- L\|d_k\|.
\end{equation}
Using $\lambda_{\min}(\nabla^2 f(x_k+d_k)) \leq -\sqrt{L\epsilon_g}$ and \eqref{eq:case2-1-H-ineq}, we obtain
$$ -2\epsilon_E - L\|d_k\| \leq -\sqrt{L\epsilon_g}. $$
This, together with $\epsilon_E= \sqrt{L\epsilon_g}/3$, implies
\begin{equation}
\label{eq:case2-dkbd}
\|d_k\| \geq  \frac{\sqrt{L\epsilon_g}-2\epsilon_E}{L}=\frac{\epsilon_E}{L}.
\end{equation}

It follows that
\begin{equation*}
    \begin{aligned}
    -m(d_{k})&\overset{\eqref{eq:case1-mdk}}{\ge}\frac{2\sigma}{3}\|d_k\|^3 + 2\epsilon_E\|d_k\|^2 - \frac{\epsilon_E^2}{L}\|d_k\|\\
    &\overset{\eqref{eq:case2-dkbd} }{\ge}\left( 0 + 2\epsilon_E\left(\frac{\epsilon_E}{L}\right) - \frac{\epsilon_E^2}{L}\right)\|d_k\|\\
    &\overset{\eqref{eq:case2-dkbd} }{\ge}\frac{\epsilon_E^3 }{L^2}.
\end{aligned}
\end{equation*}\end{itemize}
Combining (i) and (ii), we complete the proof.
\subsection{Proofs for Lemma \ref{lem:bd2}}
In this case, $d_k$ is generated by either line 13 or line 16 of Algorithm \ref{alg:cr} (Algorithm \ref{alg:ourarc}, respectively), depending on the norm of $s_k$ returned by approximately solving  \eqref{eq:mkrt} (\eqref{eq:amkrt}, respectively). Let $\sigma=L/2$ for Algorithm \ref{alg:cr} and $\sigma=\sigma_k$ for Algorithm \ref{alg:ourarc}.
We prove the results twofold.
\begin{itemize}
\item[(i)] When $\sigma\|s_k\|+\alpha_k\ge 0$, we must have $[\sigma\|s_k\| +\alpha_k]_+=\sigma\|s_k\| + \alpha_k$.
Note that
\begin{equation}
\label{eq:lem2-cri}
\nabla \tilde m_k^r(s_k)=g_k+H_ks_k+(2\epsilon_E-\alpha_k)s_k+\left[\sigma\|s_k\|+\alpha_k\right]_+ {s_k}.
\end{equation}
Using \eqref{eq:gmkrt} and $d_k=s_k$, we obtain
\begin{equation}
\label{eq:case2-2-gd-dHd}
\begin{aligned}
&g_k^{\top}d_k + d_k^{\top}H_kd_k + (2\epsilon_E-\alpha_k)\|d_k\|^2 +\left[\sigma\|d_k\|+\alpha_k\right]_+ {\|d_k\|^2} \\
\le& \left\|g_k+H_kd_k+(2\epsilon_E-\alpha_k)d_k+\left[\sigma\|d_k\|+\alpha_k\right]_+ {d_k}\right\|\cdot\|d_k\| \\
\le& \epsilon_S\|d_k\|.
\end{aligned}
\end{equation}
Since $\alpha_k\le\lambda_{\text{min}}(H_k)+\epsilon_E$, we have $H_k-\alpha_k I+2\epsilon_EI\succeq 0$ and thus
\begin{equation}
\label{eq:case2-2-psd}
d_k^{\top}H_kd_k + (2\epsilon_E-\alpha_k)\|d_k\|^2 \ge0.
\end{equation}
Noting that $\left[\sigma\|s_k\|+\alpha_k\right]_+\ge0$, we have from \eqref{eq:case2-2-gd-dHd} and \eqref{eq:case2-2-psd} that
\begin{equation}
\label{eq:case2-2-gd}
g_k^{\top}d_k \le \epsilon_S\|d_k\|.
\end{equation}
On the other hand, according to $\sigma\|s_k\|+\alpha_k\ge 0$ and \eqref{eq:case2-2-gd-dHd}, we  have
\begin{equation}
\label{eq:case2-2-dHd}
\begin{aligned}
d_k^{\top}H_kd_k &\le -g_k^{\top}d_k+(\alpha_k-2\epsilon_E)\|d_k\|^2-\sigma\|d_k\|^3-\alpha_k\|d_k\|^2+\epsilon_S\|d_k\| \\
&= -g_k^{\top}d_k - \sigma\|d_k\|^3 - 2\epsilon_E\|d_k\|^2 + \epsilon_S\|d_k\|.
\end{aligned}
\end{equation}
Since $d_k=s_k$,  and $\alpha_k\le -\epsilon_E$, from $\sigma\|s_k\|+\alpha_k\ge 0$ we obtain
\begin{equation}
\label{eq:case2-a1} \|d_k\| \geq -\frac{\alpha_k}{\sigma} \geq \frac{\epsilon_E}{\sigma}. \end{equation}
We further have
\begin{align*}
m(d_{k}) & =g_k^{\top}d_k + \frac{1}{2}d_k^{\top}H_kd_k + \frac{\sigma}{3}\|d_k\|^3 \\
& \overset{\eqref{eq:case2-2-dHd}}{\le} g_k^{\top}d_k - \frac{1}{2}g_k^{\top}d_k - \frac{\sigma}{2}\|d_k\|^3 - \epsilon_E\|d_{k}\|^2\ + \frac{1}{2}\epsilon_S\|d_k\| + \frac{\sigma}{3}\|d_k\|^3 \\
& = \frac{1}{2}g_k^{\top}d_k + \frac{1}{2}\epsilon_S\|d_k\| - \epsilon_E\|d_k\|^2  - \frac{\sigma}{6}\|d_k\|^3\\
& \overset{\eqref{eq:case2-2-gd}}{\leq} \frac{1}{2}\epsilon_S\|d_k\| + \frac{1}{2}\epsilon_S\|d_k\| - \epsilon_E\|d_k\|^2  - \frac{\sigma}{6}\|d_k\|^3\\
& = -\frac{\sigma}{6}\|d_k\|^3 - \epsilon_E\|d_k\|^2 + \epsilon_S\|d_k\|.
\end{align*}
This gives
\begin{equation}
\label{eq:case2-mdkneg}
-m(d_{k})   \geq \frac{\sigma}{6}\|d_k\|^3 + \epsilon_E\|d_k\|^2 - \epsilon_S\|d_k\|.
\end{equation}

Thus the desired bound holds immediately for Algorithm \ref{alg:cr}.

Now consider Algorithm \ref{alg:ourarc}.
Similar to the proof for Lemma \ref{lem:bd1}, we consider two cases  $\|\nabla f(x_k+d_k)\|>\epsilon_g$, and $\|\nabla f(x_k+d_k)\|\le \epsilon_g,~\lambda_{\min}(\nabla^2 f(x_k+d_k)) < -\sqrt{L\epsilon_g}$. For the latter case, similar to case (ii) in the proof of Lemma \ref{lem:bd1}, from Lipschitz continuity of Hessian, we have $\|d_k\|\ge \epsilon_E/L$. Then due to $\epsilon_S=\epsilon_E^2/L$ in Condition \ref{con:eps}, we have $2\epsilon_E\|d_k\|-\epsilon_S\ge 0$ and thus \eqref{eq:case2-mdkneg} yields
\[
-m(d_{k})   \geq \frac{\sigma}{6}\|d_k\|^3\overset{\eqref{eq:case2-a1}}{\ge} \frac{\epsilon_E^3}{6\sigma^2}.
\]
Now consider the case $\|\nabla f(x_k+d_k)\|>\epsilon_g$, whose proof follows a similar idea to that in Lemma \ref{lem:bd1}. Since the subproblem is approximately solved, we have
\[
\|\nabla \tilde m_k^r(d_k)\|\le \epsilon_S
\]
and thus  \eqref{eq:lem2-cri}, together with $\sigma\|d_k\|+\alpha_k\ge0$,  gives
\[
\|g_k+H_kd_k+(2\epsilon_E-\alpha_k)d_k+\left[\sigma\|d_k\|+\alpha_k\right]_+ {d_k}\|
=\|g_k+H_kd_k + (2\epsilon_E+ \sigma\|d_k\|) {d_k}\|\le \epsilon_S.
\]
Hence  we have
\[
\epsilon_g\le \|\nabla f(x_k+d_k)\| \le \|g_k+H_k\| +\frac{L}{2}\|d_k\|^2\le \epsilon_S+2\epsilon_E\|d_k\|+(\sigma+\frac{L}{2})\|d_k\|^2.
\]
The due to Condition \ref{con:eps}, the above quadratic inequality gives
\[
\|d_k\|\ge \frac{-2\epsilon_E+\sqrt{4\epsilon_E^2-4(\sigma+L/2)(\epsilon_S-\epsilon_g)}}{2\sigma+L}=\frac{-2\epsilon_E+\sqrt{20\epsilon_E^2+32\sigma\epsilon_E^2/L}}{2\sigma+L}.
\]
We claim the following inequality holds
\begin{equation}
\label{eq:keyeq2}
\frac{\sigma}{6}\|d_k\|^2+ \epsilon_E\|d_k\| - \epsilon_S\ge\frac{\sigma}{6}\left(\frac{-2\epsilon_E+\sqrt{20\epsilon_E^2+32\sigma\epsilon_E^2/L}}{2\sigma+L}\right)^2+ \epsilon_E\frac{-2\epsilon_E+\sqrt{20\epsilon_E^2+32\sigma\epsilon_E^2/L}}{2\sigma+L}- \frac{\epsilon_E^2}{L}\ge \frac{\epsilon_E^2}{3L},
\end{equation}
which further gives
\[
-m(d_{k})   \geq \frac{1}{3L}\epsilon_E^2\|d_k\| \overset{\eqref{eq:case2-a1}}{\ge} \frac{1}{3\sigma L}\epsilon_E^3.
\]
Indeed, \eqref{eq:keyeq2} is equivalent to, by defining $a=\sigma/L$,
\[
\psi(a)=4a\sqrt{5+8a}-8a+3\sqrt{5+8a}-5\ge0,
\]
which holds since \[\frac{1}{8}\psi'(a)=\frac{6a+4}{\sqrt{5+8a}}-1 \ge0 ~\Longleftrightarrow~36a^2+40a+9\ge0 ~\Longleftrightarrow~9(2a+1)^2+4a\ge0 ~\forall a\ge0,
\]
 and $\psi(0)=3\sqrt5-5>0.$
\item[(ii)] When $\sigma\|s_k\|+\alpha_k< 0$, we have
$d_k=\frac{ 1}{2\sigma}w_k,~ \alpha_k=v_k^{\top}H_kv_k=\frac{w_k^{\top}H_kw_k}{w_k^{\top}w_k},~\|w_k\|=|\alpha_k|$, and $w_k^{\top}g_k\le 0 $.
 It follows that
\begin{equation}
    \label{eq:case2-3-dHd}
    d_k^{\top}H_kd_k=\frac{1}{4\sigma^2}w_k^{\top}H_kw_k=\frac{1}{4\sigma^2}\alpha_k w_k^{\top} w_k=\frac{1}{4\sigma^2}\alpha_k^3.
\end{equation}
Since $\alpha_k\le-\epsilon_E<0$, we also have
\begin{equation}
    \label{eq:case2-3-d}
    \|d_k\|^3=\frac{1}{8\sigma^3}\|w_k\|^3=\frac{1}{8\sigma ^3}|\alpha_k|^3=-\frac{1}{8\sigma ^3}\alpha_k^3.
\end{equation}
Then we have
\begin{align*}
m(d_{k}) &  = g_k^{\top}d_k + \frac{1}{2}d_k^{\top}H_kd_k + \frac{\sigma}{3}\|d_k\|^3 \\
& = \frac{1}{2\sigma}g_k^{\top}w_k + \frac{1}{8\sigma^2}\alpha_k^3  - \frac{1}{24\sigma^2}\alpha_k^3\\
& \leq \frac{1}{12\sigma^2}\alpha_k^3,
\end{align*}
where the second equality follows from \eqref{eq:case2-3-dHd} and \eqref{eq:case2-3-d}, and the inequality follows from $w_k^{\top}g_k\le 0$.
Due to $\alpha_k\le -\epsilon_E$, we have
$$ -m(d_{k})   \ge -\frac{1}{12\sigma^2}\alpha_k^3\ge \frac{1}{12\sigma^2}\epsilon_E^3.$$
\end{itemize}
Combining (i) and (ii) and noting $\sigma=L/2$ in Algorithm \ref{alg:cr} and $\sigma=\sigma_k\in (0,\max\{\sigma_0,\gamma L/2\})$ (due to Lemma \ref{lem:gamma}) and $\gamma>1$ in Algorithm \ref{alg:ourarc}, we complete the proof.

\bibliographystyle{unsrt}

\end{document}